\documentclass[a4paper,reqno,10pt]{amsart}

\usepackage{epsf,graphicx,epsfig}
\usepackage{amscd}
\usepackage{amsmath,latexsym,amssymb,amsthm}
\usepackage[nospace,noadjust]{cite}
\usepackage{textcomp}
\usepackage{setspace,cite}
\usepackage{lscape,fancyhdr,fancybox}
\usepackage{stmaryrd}
\usepackage[all,cmtip]{xy}
\usepackage{tikz}
\usepackage{cancel}
\usetikzlibrary{shapes,arrows,decorations.markings}
\usepackage{graphicx}

\usepackage{color}
\usepackage{url}
\usepackage{enumerate}
\usepackage[mathscr]{euscript}
  \theoremstyle{plain}

\swapnumbers
    \newtheorem{thm}{Theorem}[section]
    \newtheorem{prop}[thm]{Proposition}

    \newtheorem{subsec}[thm]{}
\theoremstyle{definition}
    \newtheorem{defn}[thm]{Definition}
        \newtheorem{remark}[thm]{Remark}
    \newtheorem{exam}[thm]{Example}

\theoremstyle{remark}

\title{}
\author{}
\date{}
\usepackage{amssymb}

\usepackage{hyperref}
\hypersetup{
	colorlinks,
	citecolor=blue,
	filecolor=black,
	linkcolor=blue,
	urlcolor=black
}

\begin{document}

\title{Averaging algebras of any nonzero weight}

\author{Apurba Das}
\address{Department of Mathematics,
Indian Institute of Technology, Kharagpur 721302, West Bengal, India.}
\email{apurbadas348@gmail.com, apurbadas348@maths.iitkgp.ac.in}

\author{Ramkrishna Mandal}
\address{Department of Mathematics, Indian Institute of Technology, Kharagpur 721302, West Bengal, India.}
\email{ramkrishnamandal430@gmail.com}

\begin{abstract}
In this paper, we first introduce the notion of a (relative) averaging operator of any nonzero weight $\lambda$. We show that such operators are intimately related to triassociative algebras introduced by Loday and Ronco. Next, we construct a differential graded Lie algebra and a $L_\infty$-algebra whose Maurer-Cartan elements are respectively relative averaging operators of weight $\lambda$ and relative averaging algebras of weight $\lambda$. Subsequently, we define cohomology theories for relative averaging operators and relative averaging algebras of any nonzero weight. Our cohomology is useful to study infinitesimal deformations of relative averaging algebras. Finally, we introduce and study homotopy triassociative algebras and homotopy relative averaging operators of any nonzero weight $\lambda$.

\end{abstract}

\maketitle

%\curraddr{}
%\email{}

%\subjclass[2010]{}
%\keywords{}

\medskip



%\medskip

%\noindent {\sf Date of resubmission:} July 26, 2021.

\thispagestyle{empty}

\tableofcontents

%\vspace{0.2cm}

\medskip

\section{Introduction}
\subsection{Averaging algebras}
The notion of averaging operator was first implicitly considered by O. Reynolds in 1895 \cite{rey}. His description was given in terms of idempotent Reynolds operators that appeared in the turbulence theory of fluid dynamics. In 1930, Kamp\'{e} de F\'{e}riet gave the explicit description of an averaging operator in the context of turbulence theory and functional analysis \cite{kampe}. An averaging operator generalizes the time average operator of real-valued functions defined in time space. Let $A$ be an associative algebra. A linear map $P: A \rightarrow A$ is said to be an averaging operator on the algebra $A$ if
\begin{align*}
    P(a)  P(b) = P ( P(a)  b) = P (a  P(b)), \text{ for } a, b \in A.
\end{align*}
An associative algebra $A$ equipped with a distinguished averaging operator on it is called an averaging algebra. In the last century, averaging operators and averaging algebras were mostly studied in the algebra of functions on a space \cite{birkhoff,kelly,gam-mill,mill,moy}. In the algebraic contexts, B. Brainerd \cite{brain} first studied the conditions for which an averaging operator can be realized as a generalization of the integral operator on the ring of real-valued measurable functions. In his PhD thesis, W. Cao \cite{cao} considered averaging operators from algebraic and combinatorial points of view. Among others, he described the Lie algebra and the Leibniz algebra induced from an averaging operator and construct the free commutative averaging algebra. Further algebraic study of averaging operators on any binary operad and their relations with bisuccessors, duplicators and Rota-Baxter operators are explicitly given in \cite{bai-imrn,pei-rep,pei-rep2}. Recently, J. Pei and L. Guo \cite{pei-guo} constructed the free (associative) averaging algebra using a class of bracketed words, called averaging words. It is worth mentioning that averaging operators also appeared in the context of Lie algebras with the name of embedding tensors. They have close connections with Leibniz algebras, tensor hierarchies and higher gauge theories \cite{bon,lavau,strobl,sheng-embedd}. 
%Very recently, cohomology and deformation theory of averaging algebras and embedding tensors are extensively studied in the literature \textcolor{red}{3+}.

\subsection{Diassociative algebras and triassociative algebras} In his study of Leibniz algebras, J.-L. Loday introduced the notion of a diassociative algebra (also called an associative dialgebra or simply a dialgebra) \cite{loday-di}. A diassociative algebra is a vector space $D$ equipped with two linear maps $\dashv, \vdash : D \otimes D \rightarrow D$ satisfying
\begin{align*}
 (x \dashv y) \dashv z =~& x \dashv (y \dashv z),\\
 (x \dashv y) \dashv z =~& x \dashv (y \vdash z),\\
 (x \vdash y) \dashv z =~& x \vdash (y \dashv z),\\
 (x \dashv y) \vdash z =~& x \vdash (y \vdash z),\\
 (x \vdash y) \vdash z =~& x \vdash (y \vdash z),
\end{align*}
for $x, y, z \in D.$ It turns out that both the operations $\dashv$ and $\vdash$ are associative. See \cite{frab,mm1,mm2,gon}  for some interesting studies on diassociative algebras. In \cite{aguiar} M. Aguiar first observes that an averaging operator $P: A \rightarrow A$ on an associative algebra $A$ induces a diassociative algebra structure $ a \dashv b := a  P(b)$  and $ a \vdash b : = P(a)  b$
on the underlying vector space $A$. However, the converse need not be true. In \cite{das} the author considered a notion of relative averaging algebra (a generalization of an averaging algebra) and showed that any diassociative algebra is always induced by a relative averaging algebra.

In \cite{loday-ronco} J.-L. Loday and M. Ronco introduced the notion of a triassociative algebra that consist of three associative operations satisfying additional compatibilities (cf. Theorem \ref{triass-defn}). It turns out that any triassociative algebra is intrinsically a diassociative algebra. In the same article, Loday and Ronco introduced tridendriform algebras and showed that triassociative algebras are Koszul dual to tridendriform algebras.
%Subsequently, (co)homology and deformations of triassociative algebras are considered by D. Yau \cite{yau,yau2}.

\subsection{Averaging operators and averaging algebras of nonzero weight}
Our first aim in the present paper is to introduce the concept of an averaging operator of any nonzero weight. Let $A$ be an associative algebra. We define an averaging operator of any nonzero weight $\lambda \in \mathbb{K}$ as a linear map $P: A \rightarrow A$ that satisfies
\begin{align*}
    P(a)  P(b) = P (P(a)  b) = P (a  P(b)) = \lambda P (a b), \text{ for } a, b \in A.
\end{align*}
It turns out that any averaging operator of weight $\lambda \neq 0$ is intrinsically an averaging operator. However, the converse need not be true (see Remark \ref{converse}). An associative algebra equipped with an averaging operator of nonzero weight $\lambda$ is called an averaging algebra of weight $\lambda$. We also introduce a generalization, called relative averaging algebra of weight $\lambda$. We show that any relative averaging algebra of weight $\lambda$ induces a triassociative algebra structure. We also consider the notion of a relative averaging algebra of weight $\lambda$ and show that any triassociative algebra is always induced by a relative averaging algebra of weight $1$. Hence we obtain the following diagram

\begin{align}
\xymatrix{
\text{Diassociative algebra} \ar[r]  & \text{(relative) averaging algebra} \ar[l] \\
\text{Triassociative algebra} \ar[r] \ar@{^{(}->}[u]^{\text{forgetful}} \ar[r] & \text{(relative) averaging algebra with weight.} \ar[l] \ar@{^{(}->}[u]_{\text{forgetful}}
}
\end{align}

\medskip

In the last twenty years, Rota-Baxter operators with weight were paid much attention due to their close relationships with combinatorics, Yang-Baxter equations and tridendriform algebras \cite{bai-imrn,fard}. A linear map $R: A \rightarrow A$ on an associative algebra $A$ is a Rota-Baxter operator of weight $\lambda$ if it satisfies
\begin{align*}
    R(a) R(b) = R( R(a) b + a R(b) + \lambda ab), \text{ for } a, b \in A.
\end{align*}
It turns out that any (relative) averaging operator of weight $\lambda$ is a (relative) Rota-Baxter operator of weight $-\lambda$. Therefore, a triassociative algebra naturally yields a tridendriform algebra structure that makes the following diagram commutative
\begin{align}\label{tri-tri-diag}
\xymatrix{
 \text{(relative) averaging algebra with weight } 1 \ar[r] \ar[d] & \text{(relative) Rota-Baxter with weight} -1  \ar[d] \\
 \text{Triassociative algebra} \ar[u] \ar@{..>}[r]  & \text{Tridendriform algebra.} \ar[u]
}
\end{align}

\subsection{Cohomology and deformations of averaging algebras of any nonzero weight} Cohomology is invariant for a given algebraic structure. The first instance of cohomology appeared in the work of Hochschild for associative algebras and later found important applications in deformation theory developed by Gerstenhaber \cite{gers}. Subsequently, their study was generalized to other algebraic structures (e.g. Lie algebras, Leibniz algebras) \cite{nij-ric,bala-leib}. See also \cite{mm2,yau2,loday-di,loday-ronco} for cohomology and deformations of diassociative algebras and triassociative algebras. 

In the present paper, we first construct a differential graded Lie algebra that characterizes relative averaging operators of nonzero weight $\lambda$ as its Maurer-Cartan elements. Subsequently, we define the cohomology of a relative averaging operator of weight $\lambda$. When one wants to define the cohomology of a relative averaging algebra of weight $\lambda$, the theory becomes much more delicate. To overcome this situation, we first use Voronov's derived bracket to construct a $L_\infty$-algebra whose Maurer-Cartan elements are precisely relative averaging algebras of any fixed weight $\lambda$. Then by a suitable twisting, we find the controlling $L_\infty$-algebra of a given relative averaging algebra of weight $\lambda$. Next, we define the cohomology of a given relative averaging algebra of weight $\lambda$ and as an application, we study infinitesimal deformations.

%When one wants to study cohomology and deformations for a \textcolor{red}{Rota-Baxter algebra} or an averaging algebra, the theory becomes much more delicate as Rota-Baxter algebras or averaging algebras are not quadratic. To overcome such situations, many authors recently considered the controlling structures (which turns out a $L_\infty$-algebra) for \textcolor{red}{Rota-Baxter algebras} or averaging algebras.

\subsection{Homotopy algebras}
Algebraic structures often appeared with their homotopy version. The concept of a strongly homotopy associative algebra (also called $A_\infty$-algebra) appeared in the work of J. Stasheff in the study of topological loop spaces \cite{stas}. Subsequently, various other homotopy algebraic structures were studied in the literature. Recently, some authors have considered homotopy averaging operators \cite{das,sheng-embedd}. In the present paper, we first explicitly define the notion of a $Triass_\infty$-algebra (strongly homotopy triassociative algebra) and show that any $Triass_\infty$-algebra is intrinsically a strongly homotopy diassociative algebra. Next, we define the notion of a homotopy relative averaging operator of weight $\lambda$ and show that such a homotopy operator induces a $Triass_\infty$-algebra structure.

\medskip

The paper is organized as follows. In Section \ref{sec2}, we recall some necessary background on triassociative algebras. In particular, we recall the graded Lie algebra whose Maurer-Cartan elements correspond to triassociative algebra structures on a given vector space. In Section \ref{sec3}, we introduce averaging operators and averaging algebras of any nonzero weight and construct the free averaging algebra of weight $\lambda$. In Section \ref{sec4}, we consider relative averaging operators of weight $\lambda$, their Maurer-Cartan characterization and define their cohomology theory. In Section \ref{sec5}, we study relative averaging algebras of any nonzero weight. In particular, we find out their relations with triassociative algebras. We also construct the $L_\infty$-algebra that characterizes relative averaging algebras of weight $\lambda$ as its Maurer-Cartan elements. In Sections \ref{sec6} and \ref{sec7}, we respectively study cohomology and deformations of a relative averaging algebra of weight $\lambda$. Finally, in Section \ref{sec8}, we introduce homotopy relative averaging operators of weight $\lambda$ and show that they induce homotopy triassociative algebra structures.

\medskip

All (graded) vector spaces, linear maps and tensor products are over a field $\mathbb{K}$ of characteristic $0$ unless specified otherwise.

\section{Triassociative algebras}\label{sec2}
Triassociative algebras are given by three associative multiplications satisfying some compatibilities. They were introduced by Loday and Ronco in the study of algebraic structures associated with planar trees. Here we recall the graded Lie algebra that plays an important role to study triassociative algebras.

\begin{defn}\label{triass-defn}
A {\bf triassociative algebra} is a vector space $D$ equipped with three linear maps $\dashv, \vdash, \perp : D \otimes D \rightarrow D$ satifying the following set of identities: for $x, y, z \in D$, 
\begin{align}
 (x \dashv y) \dashv z =~& x \dashv (y \dashv z),\label{a1}\\
 (x \dashv y) \dashv z =~& x \dashv (y \vdash z),\label{a2}\\
 (x \vdash y) \dashv z =~& x \vdash (y \dashv z),\label{a3}\\
 (x \dashv y) \vdash z =~& x \vdash (y \vdash z),\label{a4}\\
 (x \vdash y) \vdash z =~& x \vdash (y \vdash z),\label{a5}\\ 
 (x \dashv y) \dashv z =~& x \dashv (y \perp z),\label{a6}\\
 (x \perp y) \dashv z =~& x \perp (y \dashv z),\label{a7}\\
 (x \dashv y) \perp z =~& x \perp (y \vdash z),\label{a8}\\
 (x \vdash y) \perp z =~& x \vdash (y \perp z),\label{a9}\\
 (x \perp y) \vdash z =~& x \vdash (y \vdash z),\label{a10}\\
 (x \perp y) \perp z =~& x \perp (y \perp z).\label{a11}
\end{align}
We denote a triassociative algebra as above by $(D, \dashv, \vdash, \perp)$ or simply by $D$.
\end{defn}

Let $(D, \dashv, \vdash, \perp)$ and $(D', \dashv', \vdash', \perp')$ be two triassociative algebras. A {\bf morphism} of triassociative algebras from $(D, \dashv, \vdash, \perp)$ to $(D', \dashv', \vdash', \perp')$ is a linear map $\psi : D \rightarrow D'$ satisfying 
\begin{align*}
\psi (x \dashv y) = \psi (x) \dashv' \psi(y), \quad \psi (x \vdash y) = \psi (x) \vdash' \psi(y) ~~~ \text{ and } ~~~ \psi (x \perp y) = \psi (x) \perp' \psi (y), \text{ for } x, y \in D.
\end{align*}
The collection of all triassociative algebras and morphisms between them forms a category, denoted by {\bf Triass}. It follows from the identities (\ref{a1})-(\ref{a5}) that any triassociative algebra is intrinsically a diassociative algebra.

%Diassociative algebras form a particular class of triassociative algebras.
%A diassociative algebra is a vector space $D$ equipped with two linear maps $\dashv,\vdash : D\otimes D \rightarrow D$ satisfying the identities (\ref{a1})-(\ref{a5}).

%Let $(D, \dashv, \vdash, \perp)$ be a triassociative algebra. A representation of the triassociative algebra is a vector space $M$ equipped with six binary operations $\dashv, \vdash, \perp:  D\otimes M \rightarrow M$ and $\dashv, \vdash, \perp:  D\otimes M \rightarrow M$ satisfying the 33 identities where each set of 11 identities correspond to the identities with exactly one of $x,y,z$ replaced by an element of $M$.

%In \cite{loday-ronco,yau2} the authors first considered the cohomology of a triassociative algebra. To recall their cohomology, we need to consider planar $n$-trees. 
A planar tree is called $n$-tree if it has $(n+1)$ leaves and one root. We denote the set of all planar $n$-trees by $T_n$, for $n \geq 1$. Let $T_0$ be the set consisting of a root only. Thus, for small values of $n$, the set $T_n$'s are given by \\

\begin{center}
$T_0$ = \Bigg\{ \begin{tikzpicture}[scale=0.15]    
\draw (0,-2) -- (0,3);
\end{tikzpicture} \Bigg\}, \qquad $T_1$ = \Bigg\{  
\begin{tikzpicture}[scale=0.15]
\draw (0,0) -- (0,-2); \draw (0,0) -- (-3,3); \draw (0,0) -- (3,3);
\end{tikzpicture}
\Bigg\}, \qquad 
$T_2$ = \Bigg\{  
\begin{tikzpicture}[scale=0.15]
\draw (0,0) -- (0,-2); \draw (0,0) -- (-3,3);  \draw (0,0) -- (3,3); \draw (1.5,1.5) -- (0,3);
\end{tikzpicture} ~,~ 
\begin{tikzpicture}[scale=0.15]
 \draw (0,0) -- (0,-2); \draw (0,0) -- (-3,3);  \draw (0,0) -- (3,3); \draw (-1.5,1.5) -- (0,3);
\end{tikzpicture} ~,~
\begin{tikzpicture}[scale=0.15]
 \draw (0,0) -- (0,-2); \draw (0,0) -- (-3,3);  \draw (0,0) -- (3,3); \draw (0,0) -- (0,3);
\end{tikzpicture}
\Bigg\}, \\ \medskip \medskip $T_3$ = \Bigg\{
 \begin{tikzpicture}[scale=0.15]
 \draw (0,0) -- (0,-2); \draw (0,0) -- (-3,3);  \draw (0,0) -- (3,3); \draw (1,1) -- (-1,3); \draw (2,2) -- (1,3);
\end{tikzpicture} ~,~ 
\begin{tikzpicture}[scale=0.15]
 \draw (0,0) -- (0,-2); \draw (0,0) -- (-3,3);  \draw (0,0) -- (3,3); \draw (1,1) -- (-1,3); \draw (0,2) -- (1,3);
\end{tikzpicture} ~,~
\begin{tikzpicture}[scale=0.15]
 \draw (0,0) -- (0,-2); \draw (0,0) -- (-3,3);  \draw (0,0) -- (3,3); \draw (-2,2) -- (-1,3); \draw (2,2) -- (1,3);
\end{tikzpicture} ~,~ 
\begin{tikzpicture}[scale=0.15]
 \draw (0,0) -- (0,-2); \draw (0,0) -- (-3,3);  \draw (0,0) -- (3,3); \draw (-1,1) -- (1,3); \draw (0,2) -- (-1,3);
\end{tikzpicture} ~,~
\begin{tikzpicture}[scale=0.15]
 \draw (0,0) -- (0,-2); \draw (0,0) -- (-3,3);  \draw (0,0) -- (3,3); \draw (-1,1) -- (1,3); \draw (-2,2) -- (-1,3);
\end{tikzpicture} ~,~
\begin{tikzpicture}[scale=0.15]
 \draw (0,0) -- (0,-2); \draw (0,0) -- (-3,3);  \draw (0,0) -- (3,3); \draw (1,1) -- (-1,3); \draw (1,1) -- (1,3);
\end{tikzpicture} ~,~
\begin{tikzpicture}[scale=0.15]
 \draw (0,0) -- (0,-2); \draw (0,0) -- (-3,3);  \draw (0,0) -- (3,3); \draw (0,0) -- (0,3); \draw (2,2) -- (1.5,3);
\end{tikzpicture} ~,~
\begin{tikzpicture}[scale=0.15]
 \draw (0,0) -- (0,-2); \draw (0,0) -- (-3,3);  \draw (0,0) -- (3,3); \draw (0,2) -- (-1,3); \draw (0,2) -- (1,3); \draw (0,0) -- (0,2);
\end{tikzpicture} ~,~ 
\begin{tikzpicture}[scale=0.15]
 \draw (0,0) -- (0,-2); \draw (0,0) -- (-3,3);  \draw (0,0) -- (3,3); \draw (0,0) -- (0,3); \draw (-2,2) -- (-1.5,3);
\end{tikzpicture} ~,~
\begin{tikzpicture}[scale=0.15]
 \draw (0,0) -- (0,-2); \draw (0,0) -- (-3,3);  \draw (0,0) -- (3,3); \draw (-1,1) -- (1,3); \draw (-1,1) -- (-1,3);
\end{tikzpicture} ~,~
\begin{tikzpicture}[scale=0.15]
 \draw (0,0) -- (0,-2); \draw (0,0) -- (-3,3);  \draw (0,0) -- (3,3); \draw (0,0) -- (-1,3); \draw (0,0) -- (1,3);
\end{tikzpicture}
 \Bigg\}, ~~~ ~~~ \text{ ~~ etc.}
\end{center}

\medskip

Let $T^1$ be a $n_1$-tree, $T^2$ be a $n_2$-tree, \ldots , $T^k$ be a $n_k$-tree. Then their grafting $T^1 \vee \cdots \vee T^k$ is a $(n_1 + \cdots + n_k + k -1)$-tree which is obtained by arranging $T^1, \ldots, T^k$ from left to right and joining their roots to a new vertex and creating a new root from that vertex.

Let $T\in T_n$ be an $n$-tree. We label the $(n+1)$ leaves of $T$ by $0,1,\ldots,n$ (from left to right). A leaf is said to be left-oriented (resp. right-oriented) if it is the left most (resp. right most) leaf of the vertex underneath it. A leaf is called a middle leaf if it is neither left-oriented nor right-oriented. Note that, for each $n\geq 1$ and $0\leq i\leq n$, there is a map $d_i:T_n\rightarrow T_{n-1}$, $T\mapsto d_i(T)$ which is obtained from $T$ by removing the $i$-th leaf of $T$. Moreover, for each $0 \leq i \leq n$, there is a map $\bullet_i : T_n \rightarrow \{ \dashv, \vdash, \perp \}, T \mapsto \bullet_i^T$ given by
\begin{align*}
    \bullet_0^T = \begin{cases}
        \dashv & \text{ if } T = | \vee T^1 \text{ for some } (n-1)\text{-tree } T^1 \\
        \perp & \text{ if } T = | \vee T^1 \vee \cdots \vee T^k \text{ for } k > 1 \\
        \vdash & \text{ otherwise},\\
    \end{cases} \quad 
      \bullet_n^T = \begin{cases}
        \vdash & \text{ if } T =  T^1 \vee | \text{ for some } (n-1)\text{-tree } T^1 \\
        \perp & \text{ if } T = T^1 \vee \cdots \vee T^k \vee | \text{ for } k > 1 \\
        \dashv & \text{ otherwise},\\
    \end{cases}
\end{align*}
and for $1 \leq i \leq n-1$,
\begin{align*}
  \bullet_i^T = \begin{cases}
        \dashv & \text{ if the } i\text{-th leaf of } T \text{ is left oriented} \\
        \vdash &  \text{ if the } i\text{-th leaf of } T \text{ is right oriented}  \\
        \perp &  \text{ if the } i\text{-th leaf of } T \text{ is a middle leaf.}\\
    \end{cases}  
\end{align*}

%Let $(D, \dashv, \vdash, \perp)$ be a triassociative algebra and $M$ be a representation. For each $n\geq 0$, we define the space $C^n(D,M)$ of $n$-cochains by $C^n(D,M)=\mathrm{Hom}(\mathbb{K}[T_n]\otimes D^{\otimes n},M)$. Then there is a map $\delta_{\mathrm{Triass}} : C^n{(D,M)}\rightarrow C^{n+1}{(D,M)}$ given by 
%\begin{align*}
%(\delta_{\mathrm{Triass}}(f))(T;x_1,\ldots,x_{n+1})
%=~& x_1 \bullet_0^T f({d_0(T);x_2,\ldots,x_{n+1}}),\nonumber \\
%&+ \displaystyle\sum_{i=1}^{n}(-1)^i~f({d_i(T);x_1,\ldots,x_{i-1},x_i\bullet_i^T x_{i+1},\ldots ,x_{n+1}}),\nonumber\\
%&+ (-1)^{n+1}f({d_{n+1}(T);x_1,\ldots,x_{n}})\bullet_{n+1}^T x_{n+1},
%\end{align*}
%for $f \in C^n(D,M),~ T \in T_{n+1}$ and $x_1,\ldots,x_{n+1}\in D$. Then it has been shown that $\{C^{\bullet}(D,M), \delta_{\mathrm{Triass}} \} $ is a cochain complex. The corresponding cohomology is called the cohomology of the triassociative algebra $D$ with coefficients in the representation $M$.

\medskip

In \cite{yau} the author considers a graded Lie algebra (associated with a vector space $D$) whose Maurer-Cartan elements correspond to triassociative algebra structures on $D$. This graded Lie algebra plays a crucial role in the whole study of the present paper. First recall that, for each $m,n\geq 1$ and $0\leq i\leq m$, there are maps $R_0^{m;i,n} : T_{m+n-1}\rightarrow T_m$ and $R_i^{m;i,n} : T_{m+n-1}\rightarrow T_n$ given by 
\begin{align}
 R_0^{m;i,n}(T) = \widehat{d_0}\circ \widehat{d_1}\circ\cdots \circ \widehat{d_{i-1}}\circ d_i\circ \cdots \circ d_{i+n-2}\circ \widehat{d_{i+n-1}}\circ \cdots \circ \widehat{d_{m+n-1}}(T),\\
 R_i^{m;i,n}(T) = d_0\circ d_1\circ\cdots \circ d_{i-2}\circ \widehat{d_{i-1}}\circ \cdots \circ \widehat{d_{i+n-1}}\circ d_{i+n}\circ \cdots \circ d_{m+n-1}(T),
\end{align}
for $T \in T_{m+n-1}$. With these building blocks, the graded space \begin{align*}\oplus_{n=1}^{\infty} C^{n}(D,D) = \oplus_{n=1}^\infty \mathrm{Hom}(\mathbb{K}[T_n] \otimes D^{\otimes n}, D)\end{align*}  can be given a degree -1 graded Lie bracket \begin{align*}
 {[f,g]} = \displaystyle\sum_{i=1}^{m} (-1)^{(i-1)(n-1)} f\circ_i g -  (-1)^{(m-1)(n-1)} \displaystyle\sum_{i=1}^{n} (-1)^{(i-1)(m-1)} g\circ_i f,
\end{align*}
where 
\begin{align}\label{tri-circ}
 (f\circ_i g)(T;x_1,\ldots,x_{m+n-1}) = f(R_0^{m;i,n}(T);x_1,\ldots,x_{i-1},g(R_i^{m;i,n}(T);x_i,\ldots,x_{i+n-1}),x_{i+n},\ldots,a_{m+n-1}),  
\end{align}
for $T \in T_{m+n-1}$ and $x_1,\ldots,x_{m+n-1}\in D$.
That is, the shifted graded space $\displaystyle \oplus_{n=0}^{\infty} C^{n+1}(D,D)$ with the bracket $[~,~]$ becomes a graded Lie algebra.

\begin{thm}
Let $D$ be a vector space. Then there is a one-to-one correspondence between triassociative algebra structures on $D$ and Maurer-Cartan elements in the graded Lie algebra $\big( \oplus_{n=0}^{\infty} C^{n+1}(D,D),[~,~]\big).$   
\end{thm}

Note that an element $\pi\in C^2(D,D)$ corresponds to three multiplications  $\dashv, \vdash, \perp:  D\otimes D \rightarrow D$ given by \begin{align*}
 x\dashv y = \pi (\begin{tikzpicture}[scale=0.10]
\draw (0,0) -- (0,-2); \draw (0,0) -- (-3,3);  \draw (0,0) -- (3,3); \draw (1.5,1.5) -- (0,3);
\end{tikzpicture} ;x,y ), \quad x\vdash y = \pi (\begin{tikzpicture}[scale=0.10]
 \draw (0,0) -- (0,-2); \draw (0,0) -- (-3,3);  \draw (0,0) -- (3,3); \draw (-1.5,1.5) -- (0,3);
\end{tikzpicture};x,y ) ~~~ \text{ and }  ~~~ x\perp y =  \pi (\begin{tikzpicture}[scale=0.10]
 \draw (0,0) -- (0,-2); \draw (0,0) -- (-3,3);  \draw (0,0) -- (3,3); \draw (0,0) -- (0,3);
\end{tikzpicture};x,y ) ,
\end{align*}
for $x,y\in D$. Then $[\pi,\pi] = 0$ if and only if the operations $\dashv, \vdash, \perp$ make $D$ into a triassociative algebra. Given a triassociative algebra $(D, \dashv, \vdash, \perp)$, one can define a map  $\delta_{\mathrm{Triass}} : C^n{(D,D)}\rightarrow C^{n+1}{(D,D)}$ by \begin{align*}
\delta_{\mathrm{Triass}}(f) = (-1)^{n-1}[\pi,f], \text{ for } f\in C^n(D,D).
\end{align*}
The map $\delta_\mathrm{Triass}$ is explicitly given by
\begin{align*}
(\delta_{\mathrm{Triass}}(f))(T;x_1,\ldots,x_{n+1})
=~& x_1 \bullet_0^T f({d_0(T);x_2,\ldots,x_{n+1}}) + (-1)^{n+1}f({d_{n+1}(T);x_1,\ldots,x_{n}})\bullet_{n+1}^T x_{n+1} \\
&+ \sum_{i=1}^{n}(-1)^i~f({d_i(T);x_1,\ldots,x_{i-1},x_i\bullet_i^T x_{i+1},\ldots ,x_{n+1}}),
\end{align*}
for $f \in C^n(D,D),~ T \in T_{n+1}$ and $x_1,\ldots,x_{n+1}\in D$. It turns out that $\{ C^\bullet (D,D), \delta_\mathrm{Triass} \}$ is a cochain complex. The corresponding cohomology is the cohomology of the triassociative algebra $D$.

\section{Averaging operators and averaging algebras of any nonzero weight}\label{sec3}
In this section, we first introduce the notion of an averaging operator of any nonzero weight $\lambda \in \mathbb{K}$. 
%An algebra equipped with an averaging operator of weight $\lambda$ is called an averaging algebra of weight $\lambda$.
We also construct free averaging algebra of any nonzero weight $\lambda$ over any nonempty set $X.$

%Subsequently, we generalize this notion and introduce relative averaging operators of nonzero weight $\lambda$. We give some characterizations of relative averaging operators of weight $\lambda$ in terms of their graph.

\begin{defn}
 Let $A$ be an associative algebra. A linear map $P: A\rightarrow A$ is said to be an {\bf averaging operator of weight $\lambda$} $(\neq 0)$  if 
 \begin{align}\label{avg-wgt-l}
 P(a) P(b) = P(P(a) b) = P(a P(b)) = \lambda P(a b), \text{ for all } a,b\in A.    
 \end{align}
\end{defn}

It follows from the above definition that an averaging operator of nonzero weight $\lambda$ is automatically an averaging operator on $A$. The above definition can be equivalently stated as follows: A linear map $P: A\rightarrow A$ is said to be an averaging operator of weight $\lambda~(\neq 0)$  if and only if $(1/\lambda) P$ is an averaging operator and an algebra morphism. This also suggests that we can concentrate only on averaging operators of weight $1$.

\begin{remark}
The above definition of an averaging operator of nonzero weight $\lambda$ doesn't get back to the usual averaging operator if we substitute $\lambda = 0$. However, the identities in (\ref{avg-wgt-l}) can be equivalently rephrased as
\begin{align*}
P(a)  P(b) = P (a P(b)) = P(P(a) b) ~~~~ \text{ and } ~~~~ \lambda P(a) P(b) = \lambda^2 P(a  b), \text{ for } a, b \in A.
\end{align*}
Thus, an averaging operator of weight $\lambda$ is an averaging operator (in the usual sense) satisfying additionally $\lambda P(a) P(b) = \lambda^2 P(a b)$, for all $a, b \in A$. With this consideration, an averaging operator on $A$ is simply an averaging operator of weight $0$. 
%In the present paper, we are mostly interested in averaging operators of nonzero weights, hence we will consider (\ref{avg-wgt-l}).
\end{remark}

\begin{remark}\label{converse}
It follows from the above definition that any averaging operator of nonzero weight $\lambda$ is intrinsically an averaging operator. However, the converse need not be true. Let $A$ be an associative algebra and $e \in \mathrm{Cent}(A)$ be any element from the centre of $A$. Then the linear map $P_e: A \rightarrow A,~ P_e (a) = ea$ is an averaging operator but in general not an averaging operator of any nonzero weight.
\end{remark}

%\begin{remark}
%If $P$ is an averaging operator of weight $\lambda~(\neq 0)$ on A then $(1/\lambda) P$ is an averaging operator of weight 1. That means we can concentrate more towards the weight 1 averaging operator. 
%\end{remark}

\begin{exam}
 Let A be an associative algebra. Then the identity map $P = \mathrm{id}_A: A\rightarrow A $ is an averaging operator of weight 1.  
\end{exam}

\begin{exam}
Let $P_1$ and $P_2$  be two commuting averaging operators of weight $1$ on an associative algebra $A$. Then their composition $P_1P_2$ is an averaging operator of the same weight $1$. In particular, any power of an averaging operator of weight $1$ is also an averaging operator of weight $1$.   
\end{exam}

\begin{exam}
 An unital algebra homomorphism $P: A \rightarrow A$ is idempotent if and only if it is an averaging operator of weight $1$.
\end{exam}

\begin{exam}
Let $A = A_0 \oplus A_1$ be a superalgebra (i.e. the multiplication satisfies $A_i  A_j \subset A_{i+j}$ for $0 \leq i, j , i+j \leq 1$ and $A_1 A_1 = 0$). Then the projection map $P: A \rightarrow A$, $(a_0, a_1) \mapsto a_0$ is an averaging operator of weight $1$.
\end{exam}

\begin{exam}
 Let $G$ be a finite group and $\mathbb{K}[G]$ be the group algebra of $G$. Then the linear map $P:\mathbb{K}[G]\rightarrow \mathbb{K}[G]$ defined by $a\mapsto \big(\displaystyle\sum_{g\in G} g\big) a$ is an averaging operator of weight $n$ (where $n$ is the order of the group $G$). To see this, we first observe that the element $z:= \sum_{g \in G} g$ is in the centre of $\mathbb{K}[G]$ and $z^2 = n z$. Hence for any $a, b \in \mathbb{K}[G]$, we have
 \begin{align*}
     P(a) P(b) = (za) (zb) = \begin{cases}
         = z ((za)b) = P(P(a) b ),\\
         = z (a(zb)) = P (a P(b)),\\
         =  z^2 (ab) = n P(ab).
     \end{cases}
 \end{align*}
\end{exam}

% $P(a P(b)) = P(a (z b)) = 
% z (a(z b)) = $ \begin{cases}
%= (z a)(z b) = P(a) P(b)\\
%= z((z\ a) b)) = P(P(a) b)
% \end{cases} \\
%and $P(a) P(b) = (z a)(z b) = z^2 (a b) = n z (a b) = n P(a b).$

\begin{exam}
Let $A$ be a commutative associative algebra with unity $1_A$. Consider the tensor product algebra $A \otimes \mathrm{Sym}(A)$, where $\mathrm{Sym}(A)$ is the symmetric algebra on $A$. Define a linear map 
\begin{align*}
P:  A \otimes \mathrm{Sym}(A) \rightarrow A  \otimes \mathrm{Sym}(A),~ P(\sum_i a_i \otimes s_i) = \sum_i 1_A \otimes a_i s_i,
\end{align*}
where $a_i s_i$ is the product in $\mathrm{Sym}(A)$. Then $P$ is an averaging operator of weight $1$.
\end{exam}
\begin{prop}
 Let $A$ be an associative algebra and $P$ be an averaging operator of weight 1 on A. Then for every $\phi\in \mathrm{Aut}(A)$, the operator $P^{\phi} = {\phi}^{-1}P\phi$ is also an averaging operator of weight 1.   
\end{prop}
\begin{proof}
For all $a,b\in A$, we have
\begin{align*}
 P^{\phi}(a P^{\phi}(b)) =  P^{\phi}\big(a ({\phi}^{-1}P\phi)(b)\big) &= {\phi}^{-1}P\phi\big(a ({\phi}^{-1}P\phi)(b)\big)\\
 &= {\phi}^{-1}P\big(\phi(a)P\phi(b)\big) \\
 &= {\phi}^{-1}\big(P\phi(a)P\phi(b)\big)\\
 &= ({\phi}^{-1}P\phi)(a) ({\phi}^{-1}P\phi)(b)\\
 &= P^{\phi}(a)P^{\phi}(b).
\end{align*}
Similarly, we can show that $
   P^{\phi}(P^{\phi}(a)b) =  P^{\phi}(a)P^{\phi}(b)$
 and $P^{\phi}(ab) = P^{\phi}(a)P^{\phi}(b).$ This completes the proof.
\end{proof}
Note that the identity (\ref{avg-wgt-l}) for $\lambda = 1$ is set-theoretic. Hence averaging operator of weight 1 is also defined on a semigroup. There is a close relationship between Rota-Baxter operators of weight 1 and averaging operators of weight 1.
\begin{prop}
Let $\mathbb{Q}[x]$ be the algebra of polynomials over $\mathbb{Q}$. Define a map $R:\mathbb{Q}[x]\rightarrow \mathbb{Q}[x]$ by $R(x^n)=\beta(n)x^{\theta(n)}$ with $\beta(n)\in {\mathbb{Q}_{> 0}} $ and $\theta(n)\in  \mathbb{N} \cup \{ 0 \}$. Then $R$ is a Rota-Baxter operator of weight 1 if and only if $\theta$ is an averaging operator of weight 1 on the semigroup $(\mathbb{N} \cup \{ 0 \},+)$  and 
\begin{align*}
\beta(m)\beta(n) = \beta(m) \beta(\theta (m) + n) =  \beta(m + \theta (n)) \beta(n)  = \beta(m+n), \text{ for all } m, n \geq 0.
\end{align*}
\end{prop}

\begin{proof}
Note that $R$ is a Rota-Baxter operator of weight $1$ if and only if
\begin{align}
\label{ra} \beta(m)\beta(n)x^{\theta (m)+\theta (n)} = \beta(m) \beta(\theta (m) + n)  x^{\theta (\theta (m)+n)} + \beta(m + \theta (n))\beta(n) x^{\theta (m+\theta (n))}  + \beta(m+n)x^{\theta(m+n)},  
\end{align}
for all $m,n \geq 0$. Since $\beta(n)\in {\mathbb{Q}_{> 0}}$, so all the coefficients in Equation (\ref{ra}) are nonzero. Hence the identity (\ref{ra}) holds if and only if
\begin{align*}
   &\theta (m)+\theta (n) = \theta (\theta (m)+n)  =  \theta (m+\theta (n))) = \theta(m+n),\\
& \beta(m)\beta(n) = \beta(m) \beta(\theta (m) + n)=  \beta(m + \theta (n)) \beta(n)  = \beta(m+n).
\end{align*}
This completes the proof.
\end{proof}

\subsection{Free averaging algebra of nonzero weight $\lambda$}  An {\bf averaging algebra of weight $\lambda$} is a pair $(A, P)$ consisting of an associative algebra $A$ equipped with an averaging operator $P: A \rightarrow A$ of weight $\lambda$. Let $(A, P)$ and $(A', P')$ be two averaging algebras of same weight $\lambda$. A morphism of averaging algebras of weight $\lambda$ from $(A, P)$ to $(A',P')$ is given by an algebra homomorphism $\varphi : A \rightarrow A'$ that satisfies $P' \circ \varphi = \varphi \circ P$. In this subsection, we define the free averaging algebra of any nonzero weight $\lambda$ over any nonempty set $X$. This construction is similar to the free averaging algebra given in \cite{pei-guo}.

\begin{defn}
Let $X$ be a nonempty set. The {\bf free averaging algebra of weight $\lambda$} over the set $X$ is an averaging algebra $(\mathcal{A}(X), \mathcal{P})$ of weight $\lambda$ with an inclusion $i: X \rightarrow \mathcal{A}(X)$ that satisfies the following universal condition: for any averaging algebra $(A, P)$ of weight $\lambda$ and a set map $f: X \rightarrow A$, there exists a unique morphism of averaging algebras $\widetilde{f} : (\mathcal{A}(X), \mathcal{P})  \rightsquigarrow (A, P)$ such that $\widetilde{f} \circ i = f.$
\end{defn}

Let $X$ be a nonempty set. Let $\mathcal{M}_0(X) = \mathcal{M}(X)$ be the free monoid on $X$ and we define $\mathcal{M}_1 (X) = \mathcal{M} (X \sqcup \lfloor  \mathcal{M}_0(X)  \rfloor )$. For any set $S$, here $\lfloor S \rfloor = \{ \lfloor s \rfloor | s \in S \}$ is a copy of $S$ and the notation $\sqcup$ stands for the disjoint union. The inclusion map $X \hookrightarrow X \sqcup \lfloor  \mathcal{M}_0(X)  \rfloor$ induces a monomorphism $i_0 :  \mathcal{M}_0(X) \hookrightarrow  \mathcal{M}_1(X)$. Through this map, we can identify $ \mathcal{M}_0(X)$ with its image in $ \mathcal{M}_1(X)$. For each $n \geq 1$, we will now construct a monoid $ \mathcal{M}_n(X)$ and an embedding $i_{n-1} :  \mathcal{M}_{n-1}(X) \hookrightarrow  \mathcal{M}_n (X)$ as follows. Suppose the monoid $ \mathcal{M}_{n-1}(X)$ has been constructed with the embedding $i_{n-2} :  \mathcal{M}_{n-2} (X) \hookrightarrow  \mathcal{M}_{n-1} (X)$. Define 
\begin{align*}
     \mathcal{M}_n(X) =  \mathcal{M} (X \sqcup \lfloor   \mathcal{M}_{n-1} (X) \rfloor).
\end{align*}
Since $ \mathcal{M}_{n-1} (X)= \mathcal{M} (X \sqcup  \lfloor   \mathcal{M}_{n-2} (X) \rfloor)$ is the free monoid on the set $X \sqcup  \lfloor   \mathcal{M}_{n-2} (X) \rfloor$, the injection map $X \sqcup  \lfloor   \mathcal{M}_{n-2} (X) \rfloor \hookrightarrow X \sqcup  \lfloor   \mathcal{M}_{n-1} (X) \rfloor$ induces an embedding
\begin{align*}
\mathcal{M}_{n-1} (X) = \mathcal{M} (X \sqcup \lfloor   \mathcal{M}_{n-2} (X) \rfloor) \xrightarrow{i_n} \mathcal{M} (X \sqcup \lfloor   \mathcal{M}_{n-1} (X) \rfloor) = \mathcal{M}_n (X).
\end{align*}
We define the monoid $\overline{\mathcal{M} (X)}$ as the direct limit of the monoids $\mathcal{M}_0 (X) \subset \mathcal{M}_1 (X) \subset \cdots \subset \mathcal{M}_n (X) \subset \cdots .$ That is, $ \overline{\mathcal{M} (X)} =  \lim_{n\to\infty} \mathcal{M}_n (X) = \bigcup_{n=0}^\infty \mathcal{M}_n (X)$. The elements of $\overline{\mathcal{M}(X)}$ are often called bracketed words or bracketed monomials on $X$. Let $\mathbb{K} [\overline{\mathcal{M}(X)}]$ be the free vector space over the set $\overline{\mathcal{M}(X)}$. Note that the monoid product (concatenation product) on $\overline{\mathcal{M}(X)}$ can be extended to a multiplication on $\mathbb{K} [\overline{\mathcal{M}(X)}]$ that makes $\mathbb{K} [\overline{\mathcal{M}(X)}]$ into an associative algebra. On the other hand, the map $\lfloor ~ \rfloor: \overline{\mathcal{M}(X)} \rightarrow \overline{\mathcal{M}(X)}$, $x \mapsto \lfloor x \rfloor$, for $x \in \overline{\mathcal{M}(X)}$ can be extended to a linear map (which we denote by the same notation) $\lfloor ~ \rfloor: \mathbb{K} [\overline{\mathcal{M}(X)}] \rightarrow \mathbb{K} [\overline{\mathcal{M}(X)}]$. Define a subset $\mathcal{S} \subset \mathbb{K} [\overline{\mathcal{M}(X)}]$ by
\begin{align*}
    \mathcal{S} = \big\{  \lfloor x \rfloor \lfloor y \rfloor - \lfloor \lfloor x \rfloor y \rfloor,~ \lfloor x \rfloor \lfloor y \rfloor - \lfloor x \lfloor y \rfloor \rfloor,~ \lfloor x \rfloor \lfloor y \rfloor - \lambda \lfloor x y \rfloor ~|~ x, y \in  \overline{\mathcal{M}(X)}  \big\}.
\end{align*}
Let $I$ be the smallest operated ideal of $\mathbb{K} [\overline{\mathcal{M}(X)}]$ containing the set $\mathcal{S}$. Define $\mathcal{A}(X) = \mathbb{K} [\overline{\mathcal{M}(X)}] / I$ and a linear map $\mathcal{P} : \mathcal{A}(X) \rightarrow \mathcal{A}(X)$ induced by the map $\lfloor ~ \rfloor$. Then it turns out that $(\mathcal{A}(X), \mathcal{P})$ is an averaging algebra of weight $\lambda$.

Let $i : X \hookrightarrow \mathcal{A}(X)$ be the composition of the maps
\begin{align*}
X \hookrightarrow \mathcal{M}_0 (X) \subset \mathcal{M}_1 (X) \subset \cdots \subset \mathcal{M}_n (X) \subset \cdots \subset \cup_{n=1}^\infty \mathcal{M}_n (X) = \overline{\mathcal{M}(X)} \subset \mathbb{K} [ \overline{\mathcal{M}(X)}  ] \xrightarrow{q} \mathcal{A}(X).
\end{align*}
Let $(A, P)$ be any relative averaging algebra of weight $\lambda$ and $f: X \rightarrow A$ be a set map. We define a map $\overline{f} : \overline{\mathcal{M}(X)} \rightarrow A$ as follows. For $w = x_1 \cdots x_m \in \mathcal{M}_0 (X) \subset \overline{\mathcal{M}(X)}$, where $x_i \in X~(1 \leq i \leq m)$, we define
$\overline{f} (w) = f(x_1) \cdots f(x_m)$. Since $\mathcal{M}_1 (X)$ is the free monoid on the set $X \sqcup \lfloor \mathcal{M}_0 (X) \rfloor$, we can define $\overline{f}$ on the elements of $\mathcal{M}_1 (X)$. Suppose we are able to define $\overline{f}$ on the elements of $\mathcal{M}_{n-1} (X)$. As $\mathcal{M}_n (X)$ is the free monoid on the set $X \sqcup \lfloor \mathcal{M}_{n-1} (X) \rfloor$, we can define $\overline{f}$ on the elements of $\mathcal{M}_n (X)$. It is easy to see that the map $\overline{f} : \overline{\mathcal{M}(X)} \rightarrow A$ induces a map $\widetilde{f} : \mathcal{A}(X) \rightarrow A$. The map $\widetilde{f}$ is obviously an algebra morphism that satisfies $\widetilde{f} \circ i = f$ and $P \circ \widetilde{f} = \widetilde{f} \circ \mathcal{P}$. This shows that $(\mathcal{A}(X), \mathcal{P})$ is a free averaging algebra of weight $\lambda$ over the set $X$.

%********************************************************************************************************

\section{Relative averaging operators of nonzero weight}\label{sec4}
In this section, we consider a notion of the relative averaging operator of weight $\lambda$ as a generalization of the averaging operator of weight $\lambda$. We give some characterizations of such operators in terms of their graphs. Moreover, given some suitable structures, we construct a differential graded Lie algebra whose Maurer-Cartan elements are precisely relative averaging operators of nonzero weight $\lambda$. Using this characterization, we also define the cohomology associated with a relative averaging operator of weight $\lambda$.

Let $A$ be an associative algebra. An associative $A$-bimodule is an associative algebra $B_\nu$ (i.e., $\nu : B \otimes B \rightarrow B, (x, y) \mapsto xy$ is a linear map satisfying the associativity) equipped with two linear maps $l: A \otimes B \rightarrow B$, $(a,x) \mapsto a \cdot x$ and $r: B \otimes A \mapsto B$, $(x,a) \mapsto x \cdot a$, satisfying for $a, b \in A$ and $x, y \in A$,
\begin{align}
(a b) \cdot x =~& a \cdot ( b \cdot x), \qquad (a \cdot x) \cdot b = a \cdot (x \cdot b), \qquad (x \cdot a) \cdot b = x \cdot (ab), \\
(x y) \cdot a =~& x (y \cdot a), \qquad (x \cdot a) y = x (a \cdot y), \qquad (a \cdot x) y = a \cdot ( x y).
\end{align}

In this case, we often say that the associative algebra $A$ acts on the algebra $B_\nu$ and $l, r$ being the left and right $A$-actions on $B_\nu$. We denote an associative $A$-bimodule as above by $B_\nu^{l, r}$ or simply by $B$ when the associative multiplication $\nu$ of $B$ and the left and right $A$-actions on $B_\nu$ are clear from the context. It follows from the above definition that any associative algebra $A$ can be realized as an associative $A$-bimodule called the adjoint $A$-bimodule.

\begin{defn}
 Let $A$ be an associative algebra and $B$ be an associative $A$-bimodule. A linear map  $P: B\rightarrow A $ is said to be a {\bf relative averaging operator of nonzero weight} $\lambda\in \mathbb{K}$ if
 \begin{align*}
 P(x) P(y) = P(P(x)\cdot y) = P(x\cdot P(y)) = \lambda P(x y), \text{ for all } x,y\in B.    
 \end{align*}
\end{defn}
It follows that any averaging operator of weight $\lambda$ on an associative algebra $A$ can be realized as a relative averaging operator of weight $\lambda$. Thus, a relative averaging operator of weight $\lambda$ is a generalization of an averaging operator of weight $\lambda$.

%\begin{remark}
%If $P: B \rightarrow A$ is a relative averaging operator of weight $\lambda$ then one can check that the following operations $x \dashv y:= x\cdot P(y)$ and $x \vdash y:= P(x)\cdot y$ on the vector space $B$ are also associative. See Proposition \ref{ravg-ind-triass} for more details.
%\end{remark}

\begin{exam}
    Let $A$ be an associative algebra. Note that the space $B = \underbrace{A\oplus\cdots\oplus A}_{n \text{ summand}}$ can be given an associative algebra structure with the multiplication given by the componentwise multiplication of $A$. The space $B$ is an associative $A$-bimodule with the left and right $A$-actions on $B$ are respectively given by
    \begin{align}\label{left-right-b}
  a\cdot(a_1,\ldots,a_n) = (a a_1,\ldots,a a_n)~~ \text{ and } ~~  (a_1,\ldots,a_n)\cdot a = (a_1 a,\ldots,a_n a),
\end{align} 
for $a\in A$  and $(a_1,\ldots,a_n)\in B$. Then for any $1\leq i\leq n$, the $i$-th projection map $P_i : B \rightarrow A$ is a relative averaging operator of weight $1$.
\end{exam}

\begin{exam}
Let $A$ be an associative algebra. Then the space $B = \underbrace{A\oplus\cdots\oplus A}_{n \text{ summand}}$  can be equipped with an associative algebra structure with the multiplication given by
\begin{align*}
    (a_1, \ldots, a_n) (b_1, \ldots, b_n) = (a_1b_1, a_2 b_1 + a_1 b_2 + a_2b_2, \ldots, \underbrace{ \sum_{1 \leq j \leq k-1} a_k b_j + \sum_{ 1 \leq i \leq k-1} a_i b_k + a_k b_k }_{k\text{-th place}}, \ldots),
\end{align*}
for $(a_1, \ldots, a_n), (b_1, \ldots, b_n) \in B$.
Moreover, $B$ is an associative $A$-bimodule with the left and right $A$-actions on $B$ are given by (\ref{left-right-b}).
 With these notations, the usual average map $P: B\rightarrow A$, $P(a_1,\ldots,a_n) = {\frac{a_1 +\cdots + a_n}{n}}$ is a relative averaging operator of weight $1/n$.
\end{exam}

%\begin{exam}
%\textcolor{red}{ With the notations of the previous example, the $i$-th projection map $P_i: B\rightarrow A$ (for any $1\leq i\leq n)$  is a relative averaging operator of weight 1.} 
%\end{exam}

\begin{exam}
Let $A$ be an associative algebra and $B$ be an associative $A$-bimodule. Suppose $f: B\rightarrow A$ is an algebra homomorphism and an $A$-bimodule map. Then $f$ is a relative averaging operator of weight 1.   \end{exam}

\begin{exam}
Let $A$ be any associative algebra. Then the associative multiplication of $A$ induces an associative multiplication on the space $B= A[[t]]/(t^2)$ using $t$-bilinearity. Further, $B$ can be given an associative $A$-bimodule structure with the left and right $A$-actions are respectively given by
\begin{align*}
a' \cdot (a + t b) = a' a + t a'b ~~~~ \text{ and } ~~~~ (a + t b) \cdot a' = aa' + t b a',
\end{align*}
for $a' \in A$ and $a+ tb \in B$. With this, the projection map $P: B \rightarrow A$ given by $P(a+ tb) = a$ is a relative averaging operator of weight $1$.
\end{exam}

\begin{exam}
Let $V \xrightarrow{f} W$ be a $2$-term chain complex. Let $A = T(W) = \mathbb{K} \oplus W \oplus W^{\otimes 2} \oplus \cdots $ be the tensor algebra (over the vector space $W$) with the concatenation product. On the other hand, the space $B = T(W) \otimes V \otimes T(W) \otimes V \otimes T(W)$ can also be given an associative algebra structure with the multiplication
\begin{align*}
   ( w_{-m} \cdots w_{-1} \otimes u \otimes \varpi_1 \cdots \varpi_n \otimes v \otimes w_1 \cdots w_p)    ( w'_{-r} \cdots w'_{-1} \otimes u' \otimes \varpi'_1 \cdots \varpi'_s \otimes v' \otimes w'_1 \cdots w'_t) \\
   = w_{-m} \cdots w_{-1} \otimes u \otimes \varpi_1 \cdots \varpi_n f(v) w_1 \cdots w_p w'_{-r} \cdots w'_{-1} f( u') \varpi'_1 \cdots \varpi'_s \otimes v' \otimes w'_1 \cdots w'_t.
\end{align*}
Moreover, $B$ can be given an associative $A$-bimodule structure with the left and right $A$-actions on $B$ given by
\begin{align*}
  ( w'_1 \cdots w'_t)  &\cdot ( w_{-m} \cdots w_{-1} \otimes u \otimes \varpi_1 \cdots \varpi_n \otimes v \otimes w_1 \cdots w_p) \\
  &= w'_1 \cdots w'_t  w_{-m} \cdots w_{-1} \otimes u \otimes \varpi_1 \cdots \varpi_n \otimes v \otimes w_1 \cdots w_p,  \\
  ( w_{-m} \cdots w_{-1} &\otimes u \otimes \varpi_1 \cdots \varpi_n \otimes v \otimes w_1 \cdots w_p) \cdot  ( w'_1 \cdots w'_t) \\
  &=  w_{-m} \cdots w_{-1} \otimes u \otimes \varpi_1 \cdots \varpi_n \otimes v \otimes w_1 \cdots w_p  w'_1 \cdots w'_t,
\end{align*}
for $ w'_1 \cdots w'_t \in A$ and $ w_{-m} \cdots w_{-1} \otimes u \otimes \varpi_1 \cdots \varpi_n \otimes v \otimes w_1 \cdots w_p \in B$. With these notations, the map $P_f : B \rightarrow A$ defined by
\begin{align*}
    P_f (   w_{-m} \cdots w_{-1} \otimes u \otimes \varpi_1 \cdots \varpi_n \otimes v \otimes w_1 \cdots w_p   ) =  w_{-m} \cdots w_{-1} f( u)  \varpi_1 \cdots \varpi_n f( v) w_1 \cdots w_p 
\end{align*}
is a relative averaging operator of weight $1$.
\end{exam}

\begin{exam}
    (This is a generalization of the previous example) Let $A$ be an associative algebra, $V$ be any vector space and $f: V \rightarrow A$ be a linear map. Define $B= A \otimes V \otimes A \otimes V \otimes A$ with the associative multiplication
    \begin{align*}
        (a \otimes u \otimes b \otimes v \otimes c) (a' \otimes u' \otimes b' \otimes v' \otimes c') = a \otimes u \otimes b f(v) c a' f(u') b' \otimes v' \otimes c',
    \end{align*}
    for $a \otimes u \otimes b \otimes v \otimes c,~ a' \otimes u' \otimes b' \otimes v' \otimes c' \in B$. Moreover, $B$ is an associative $A$-bimodule with the left and right $A$-actions
    \begin{align*}
        a' \cdot (a \otimes u \otimes b \otimes v \otimes c) = (a' \cdot a) \otimes u \otimes b \otimes v \otimes c ~~~~ \text{ and } ~~~~ (a \otimes u \otimes b \otimes v \otimes c) \cdot a' = a \otimes u \otimes b \otimes v \otimes (c \cdot a'),
    \end{align*}
    for $a' \in A$ and $a \otimes u \otimes b \otimes v \otimes c \in B$. Then the map $P_f : B \rightarrow A$, $P_f (a \otimes u \otimes b \otimes v \otimes c) = a \cdot f(u) \cdot b \cdot f(v) \cdot c$ is a relative averaging operator of weight $1$.
\end{exam}

In the following, we give a characterization of a relative averaging operator of nonzero weight in terms of the graph of the operator. We first prove the following result.

\begin{prop}\label{l-triass}
Let $A$ be an associative algebra and $B$ be an associative $A$-bimodule. Then the direct sum $A \oplus B$ inherits a triassociative algebra structure with the operations
\begin{align*}
(a,x) \dashv (b, y) := (a  b, x \cdot b), \quad (a,x) \vdash (b, y) := (a  b,  a \cdot y ) ~~ \text{ and } ~~ (a,x) \perp (b, y) := (a  b, \lambda  x  y),
\end{align*}
for $(a,x), (b, y) \in A \oplus B.$ We denote this triassociative algebra by $A \oplus_\lambda B$.
\end{prop}
\begin{proof}
 For any  $(a,x), (b, y), (c,z)\in A \oplus B$, we have 
 \begin{align*}
 ((a,x) \dashv (b, y)) \dashv (c,z) =~& (a b,x\cdot b)\dashv (c,z) = ((a b) c, (x\cdot b)\cdot c), \\
(a,x) \dashv ((b, y) \dashv (c,z)) =~& (a,x) \dashv (bc,y\cdot c) = (a( b c), x\cdot (b c)),\\
(a,x) \dashv ((b, y) \vdash (c,z)) =~& (a,x) \dashv (bc,b\cdot z) = (a( b c), x\cdot (b c)),\\
(a,x) \dashv ((b, y) \perp (c,z)) =~& (a,x) \dashv (bc, \lambda y z) = (a( bc), x\cdot (b c)).
\end{align*}
Since $A$ is associative algebra and $B$ is an associative $A$ bimodule, we have
\begin{align*}
 ((a,x) \dashv (b, y)) \dashv (c,z) =
(a,x) \dashv ((b, y) \dashv (c,z)) = (a,x) \dashv ((b, y) \vdash (c,z) ) = (a,x) \dashv ((b, y) \perp (c,z)).
\end{align*}
Similarly, we can show that 
\begin{align*}
(a,x) \vdash ((b, y) \vdash (c,z) ) =
((a,x) \vdash (b, y) ) \vdash (c,z) ) = ((a,x) \dashv (b, y) ) \vdash (c,z) = ((a,x) \perp (b, y)) \vdash (c,z).
\end{align*}
Moreover, we have
\begin{align*}
    ((a,x) \vdash (b, y)) \dashv (c,z) = ((a b) c, (a\cdot y)\cdot c)
= (a( b c), a\cdot (y\cdot c)) = (a,x) \vdash ( (b, y) \dashv (c,z)),\\
  ((a,x) \perp (b, y)) \dashv (c,z) = ((a b) c, \lambda (x y) \cdot c )
= (a( b c),\lambda x (y\cdot c) ) = (a,x) \perp ( (b, y) \dashv (c,z)).
\end{align*}
Similarly, one can show that
\begin{align*}
  ((a,x) \dashv (b, y)) \perp (c,z) =~& (a,x) \perp ( (b, y) \vdash (c,z)), \\
  ((a,x) \vdash (b, y)) \perp (c,z) =~&   (a,x) \vdash ( (b, y) \perp (c,z)), \\
  ((a,x) \perp (b, y)) \perp (c,z) =~& (a,x) \perp ((b, y) \perp (c,z)).
\end{align*}
Hence the result follows.
\end{proof}

\begin{remark}
The converse of the above proposition also holds true. More precisely, suppose $A$ and $B$ are two vector spaces equipped with the linear maps
\begin{align*}
    \mu : A \otimes A \rightarrow A, \quad \nu : B \otimes B \rightarrow B, \quad l: A \otimes B \rightarrow B ~~~ \text{ and } ~~~~ r : B \otimes A \rightarrow A. 
\end{align*}
Then we may define three maps $\dashv, \vdash, \perp : (A \oplus B) \otimes (A \oplus B) \rightarrow (A \oplus B)$ by
\begin{align*}
    (a,x) \dashv (b, y) = ( \mu (a, b) , r(x, b)), ~~~~ (a, x) \vdash (b,y) = ( \mu (a, b), l(a, y)) ~~ \text{ and } ~~ (a,x) \perp (b, y) = ( \mu (a, b), \lambda \nu (x, y)),
\end{align*}
for $(a, x), (b, y) \in A \oplus B$.
With these notations, the maps $\dashv, \vdash, \perp$ make $A \oplus B$ into a triassociative algebra if and only if $A_\mu$ is an associative algebra and $B^{l, r}_\nu$ is an associative $A_\mu$-bimodule.
\end{remark}

\begin{prop}\label{gr}
Let $A$ be an associative algebra and $B$ be an associative $A$-bimodule. Then a linear map $P: B \rightarrow A$ is a relative averaging operator of weight $\lambda$ if and only if the graph $Gr (P) = \{ ( P(x), x) |~ x \in B \}$ is a subalgebra of the triassociative algebra $A \oplus_\lambda B$. 
\end{prop}
\begin{proof}
 For any $x,y\in B$, we have \begin{align*}
 (P(x),x)\dashv (P(y),y) =~& (P(x) P(y),x\cdot P(y)), \quad
  (P(x),x)\vdash (P(y),y) = (P(x) P(y),P(x)\cdot y) \\
  & \text{ and } ~ (P(x),x) \perp (P(y),y) = (P(x) P(y),\lambda x y). 
 \end{align*}
This shows that $Gr (P)$ is a subalgebra of the triassociative algebra $A \oplus_\lambda B$ if and only if $P(x) P(y) = P(P(x)\cdot y) = P(x\cdot P(y)) = \lambda P(x y)$, that is, $P$ is a relative averaging operator of weight $\lambda$.
\end{proof}
Let $(D, \dashv, \vdash, \perp)$ be a triassociative algebra. A linear operator $\mathcal{N} : D \rightarrow D$ is said to be a \textbf{Nijenhuis operator} on the triassociative algebra if for all $x,y\in D$ and $\star =~ \dashv, \vdash, \perp$, we have
\begin{align*}
 \mathcal{N} (x)\star \mathcal{N}(y) = \mathcal{N} (\mathcal{N}(x) \star y + x \star \mathcal{N}(y) - \mathcal{N}(x \star y)\big).
% \mathcal{N} (x)\vdash \mathcal{N}(y) = \mathcal{N}\big(\mathcal{N}(x)\vdash y + x\vdash \mathcal{N}(y) - \mathcal{N}(x\vdash y)\big)\\
% \mathcal{N} (x)\perp \mathcal{N}(y) = \mathcal{N}\big(\mathcal{N}(x)\perp y + x\perp \mathcal{N}(y) - \mathcal{N}(x\perp y)\big).
\end{align*}

\begin{prop}
Let $A$ be an associative algebra and $B$ be an associative $A$-bimodule. A linear map $P: B \rightarrow A$ is a relative averaging operator of weight $\lambda$ if and only if the map 
\begin{align*}
\mathcal{N}_P : A \oplus B \rightarrow A \oplus B, ~~ \mathcal{N}_P (a, x) = \big( P(x), 0 \big)
\end{align*}
is a Nijenhuis operator on the triassociative algebra $A \oplus_\lambda B$.
\end{prop}

\begin{proof}
  For any $(a,x), (b,y) \in A \oplus B$ and $\star = ~ \dashv, \vdash, \perp$, we have 
\begin{align*}
    \mathcal{N}_P (a, x) \star \mathcal{N}_P (b, y) = (P(x) , 0) \star (P(y), 0) = ( P(x) P(y), 0).
\end{align*}
On the other hand, by a straightforward calculation, we get that
\begin{align*}
    \mathcal{N}_P \big(   \mathcal{N}_P (a,x) \star (b, y) + (a,x) \star \mathcal{N}_P (b, y) - \mathcal{N}_P (   (a,x) \star (b, y) ) \big) = \begin{cases}
        P (x \cdot P(y)) & \text{ if } \star =~ \dashv,\\
        P (P(x) \cdot y) & \text{ if } \star =~ \vdash,\\
        \lambda P(xy) & \text{ if } \star =~\perp.
    \end{cases}
\end{align*}
Hence the result follows.
\end{proof}

 % \begin{align*}
 %  \mathcal{N}_P(a,x)\dashv \mathcal{N}_P(b,y) = (P(x),0)\dashv (P(y),0) = (P(x) P(y),0)\\
 %  and,~~ \mathcal{N}_P\big(\mathcal{N}_P(a,x)\dashv (b,y) + (a,x)\dashv \mathcal{N}_P(b,y) - \mathcal{N}_P((a,x)\dashv (b,y))\big)\\
 %  = \mathcal{N}_P\big((P(x),0)\dashv (b,y) + (a,x)\dashv (P(y),0) - \mathcal{N}_P(ab,x\cdot b)\big)\\
 %  = \mathcal{N}_P\big((P(x)\cdot b,0) + (a P(y),x\cdot P(y)) - (P(x\cdot b),0)\big) = \big(P(x\cdot P(y)),0\big)
 % \end{align*} 
 % So, $\mathcal{N}_P(a,x)\dashv \mathcal{N}_P(b,y) =  \mathcal{N}_P\big(\mathcal{N}_P(a,x)\dashv (b,y) + (a,x)\dashv \mathcal{N}_P(b,y) - \mathcal{N}_P((a,x)\dashv (b,y))\big)$ holds if and only if $P(x) P(y) = P(x\cdot P(y))$. Similarly, we can show that \begin{align*}
 %      \mathcal{N}_P(a,x)\vdash \mathcal{N}_P(b,y) =  \mathcal{N}_P\big(\mathcal{N}_P(a,x)\vdash (b,y) + (a,x)\vdash \mathcal{N}_P(b,y) - \mathcal{N}_P((a,x)\vdash (b,y))\big)\\ and ~~\mathcal{N}_P(a,x)\perp \mathcal{N}_P(b,y) =  \mathcal{N}_P\big(\mathcal{N}_P(a,x)\perp (b,y) + (a,x)\dashv \mathcal{N}_P(b,y) - \mathcal{N}_P((a,x)\perp (b,y))\big) \end{align*}  holds if and only if $P(x) P(y) = P(P(x)\cdot y)$ and $P(x)P(y) = \lambda P(x y)$.\\
 %      Hence, $\mathcal{N}_P$is a Nijenhuis operator on the triassociative algebra $A\oplus_\lambda B$ if and only if $ P(x) P(y) = P(P(x)\cdot y) = P(x\cdot P(y)) = \lambda P(x y)~~for~~ all ~~x,y \in B$ i.e. $P$ is a relative averaging operator of weight $\lambda$. 
%\end{proof}

\subsection{Maurer-Cartan characterization and cohomology of relative averaging operators of any nonzero weight}

Let $A$ be an associative algebra and $B$ be an associative $A$-bimodule. Given the above data, here we construct a differential graded Lie algebra (dgLa) whose Maurer-Cartan elements correspond to relative averaging operators of weight $\lambda$. Subsequently, we define the cohomology associated with a relative averaging operator $P: B\rightarrow A$ of weight $\lambda$. First, we consider the graded Lie algebra $\mathfrak{g} = (\oplus_{n=0}^{\infty} {C}^{n+1}(A\oplus B,A\oplus B),{[~,~]} )$ associated with the direct sum vector space $A\oplus B$. Then it is easy to see that the graded subspace $\mathfrak{a} = \oplus_{n=0}^{\infty} {C}^{n+1}( B,A)$ is an abelian Lie subalgebra of $\mathfrak g$. We also define an element $\pi_\lambda \in {C}^{2}(A\oplus B,A\oplus B)$ by 
\begin{align*}
 \pi_\lambda \big(\begin{tikzpicture}[scale=0.10]
\draw (0,0) -- (0,-2); \draw (0,0) -- (-3,3);  \draw (0,0) -- (3,3); \draw (1.5,1.5) -- (0,3);
\end{tikzpicture} ;(a,x),&(b,y)\big) = (a b,x \cdot b), \quad \pi_\lambda \big(\begin{tikzpicture}[scale=0.10]
 \draw (0,0) -- (0,-2); \draw (0,0) -- (-3,3);  \draw (0,0) -- (3,3); \draw (-1.5,1.5) -- (0,3);
\end{tikzpicture};(a,x),(b,y)\big) = (a b,a\cdot y) \\
&\text{ and } ~  \pi_\lambda \big(\begin{tikzpicture}[scale=0.10]
 \draw (0,0) -- (0,-2); \draw (0,0) -- (-3,3);  \draw (0,0) -- (3,3); \draw (0,0) -- (0,3);
\end{tikzpicture};(a,x),(b,y)\big) = (a b, \lambda x y), 
\end{align*}
for $(a,x),(b,y)\in A\oplus B$. Then it follows from Proposition \ref{l-triass} that ${[\pi_\lambda,\pi_\lambda]} = 0$.
Therefore, the suspended graded vector space $s \mathfrak{a} = \displaystyle \oplus_{n=1}^{\infty} {C}^{n}( B,A)$ inherits a graded Lie algebra structure with the derived bracket (see \cite{yks})
\begin{align}\label{deri-bracket}
 \llbracket {f},{g}\rrbracket := (-1)^m {[{[\pi_\lambda ,f]},g]},   
\end{align} 
for $f \in C^m(B,A)$ and $g\in C^n(B,A)$. 
On the other hand, $- \pi_\lambda$ is a Maurer-Cartan element of the graded Lie algebra $\mathfrak{g}$ implies that $- \pi_\lambda$ induces a differential 
\begin{align*}
d_{- \pi_\lambda} := - [\pi_\lambda , - ] : C^n(A \oplus B, A \oplus B) \rightarrow C^{n+1} (A \oplus B, A \oplus B), \text{ for } n \geq 1.
\end{align*}
It can be checked that the graded subspace $\displaystyle \oplus_{n=1}^{\infty} {C}^{n}(B,A)$ is closed under the differential $d_{-\pi_\lambda }$. We denote the restriction of the differential $d_{-\pi_\lambda }$ to the subspace $\displaystyle \oplus_{n=1}^{\infty} {C}^{n}(B,A)$ simply by $d$. The map $d:{C}^{n}(B,A)\rightarrow {C}^{n+1}( B,A)$ is explicitly given by

\begin{align}\label{dgla-d-formula}
 (df)(T;x_1,\ldots,x_{n+1}) = (-1)^{(n-1)}\displaystyle \sum_{i=1}^{n} (-1)^{(i-1)} f \big( d_iT;x_1,\ldots,x_{i-1}, \pi_\lambda \big(R^{n;i,2}_i (T) ; x_i,  x_{i+1} \big) ,\ldots,x_{n+1} \big),
\end{align}
for $T\in T_{n+1}$ and $x_1,\ldots,x_{n+1}\in B$. Then $d^2=0$. Moreover, we have the following result.

\begin{prop}
 Let $A$ be an associative algebra and $B$ be an associative $A$-bimodule. Then the triple $(\displaystyle \oplus_{n=1}^{\infty} {C}^{n}( B,A),\llbracket ~,~\rrbracket,d)$ is a differential graded Lie algebra.   
\end{prop}

\begin{proof}
    For any $f \in C^m(B,A)$ and $g \in C^n (B,A)$, we have
    \begin{align*}
 d \llbracket f,g \rrbracket & = (-1)^m d{[{[\pi_\lambda ,f]},g]}\\
 &= (-1)^{m-1} {[\pi_\lambda ,{[{[\pi_\lambda ,f]},g]}]}\\
 &= (-1)^{m-1} {[{[\pi_\lambda ,{[\pi_\lambda ,f]}]},g]} + (-1)^{m-1}(-1)^m{[{[\pi_\lambda ,f]},{[\pi_\lambda,g]}]}\\
 &= (-1)^{m-1}(-1){[{[\pi_\lambda ,{[\pi_\lambda,f]}]},g]} + {[{[\pi_\lambda ,f]},dg]}\\
 &= (-1)^{m-1}{[{[\pi_\lambda ,df]},g]} + {[{[\pi_\lambda ,f]},dg]}\\
 &= \llbracket df,g\rrbracket + (-1)^m \llbracket f,dg \rrbracket .
\end{align*} 
Hence the result follows.
\end{proof}

% for any $\lambda\in \mathbb{K}$, the associative multiplication $-\lambda\nu \in \mathrm{Hom}(B^{\otimes 2},B)$ induced an element (denoted by the same notation) $-\lambda\nu \in {C}^{2}(A\oplus B,A\oplus B)$ by \begin{align*}
 %(-\lambda\nu)\big(\begin{tikzpicture}[scale=0.10]
%\draw (0,0) -- (0,-2); \draw (0,0) -- (-3,3);  \draw (0,0) -- (3,3); \draw (1.5,1.5) -- (0,3);
%\end{tikzpicture} ;(a,x),(b,y)\big) = ( -\lambda\nu) \big(\begin{tikzpicture}[scale=0.10]
% \draw (0,0) -- (0,-2); \draw (0,0) -- (-3,3);  \draw (0,0) -- (3,3); \draw (-1.5,1.5) -- (0,3);
%\end{tikzpicture};(a,x),(b,y)\big) = (0,0)~~and ~~  (-\lambda\nu)\big(\begin{tikzpicture}[scale=0.10]
 %\draw (0,0) -- (0,-2); \draw (0,0) -- (-3,3);  \draw (0,0) -- (3,3); \draw (0,0) -- (0,3);
%\end{tikzpicture};(a,x),(b,y)\big) = (0,-\lambda x y), 
%\end{align*}
%for $(a,x),(b,y)\in A\oplus B$. The element $ -\lambda\nu$ also satisfies ${[-\lambda\nu,-\lambda\nu]}_T = 0$. Hence it induces a differential  $d_{-\lambda\nu} = -{[-\lambda\nu,-]}$ on the graded vector space $\displaystyle \oplus_{n=1}^{\infty} {C}^{n}(A\oplus B,A\oplus B)$. 

%\begin{proof}
%We have already observed that $\llbracket ~,~\rrbracket$ is a graded Lie bracket and $d$ is a differential on the graded vector space ${C}^{n}( B, A)$. Thus, we only need to prove the compatibility between them. For this, we first observe that the elements $\pi$ and $\lambda\nu \in {C}^{2}(A\oplus B,A\oplus B) $ satisfies the compatibility ${[\pi,\lambda\nu]} = 0$. Hence, for any $P\in {C}^{m}( B,A)$ and $Q\in {C}^{n}( B,A)$, we have 
%\end{proof}

In the following result, we show that the above differential graded Lie algebra characterizes relative averaging operators of weight $\lambda$ as its Maurer-Cartan elements.

\begin{thm}\label{dgla-thmm}
 Let $A$ be an associative algebra and $B$ be an associative $A$-bimodule. A linear map $P: B\rightarrow A$ is a relative averaging operator of weight $\lambda$ if and only if  $P\in C^1(B,A) = \mathrm{Hom}(\mathbb{K}[T_1]\otimes B, A)$ is a Maurer-Cartan element of the differential graded Lie algebra $(\displaystyle \oplus_{n=1}^{\infty} {C}^{n}( B,A),\llbracket ~,~\rrbracket,d ).$
 \end{thm}

\begin{proof}
For any $P \in C^1 (B, A)$, it follows from (\ref{deri-bracket}) that
\begin{align*}
    \llbracket P, P \rrbracket \big(\begin{tikzpicture}[scale=0.10]
\draw (0,0) -- (0,-2); \draw (0,0) -- (-3,3);  \draw (0,0) -- (3,3); \draw (1.5,1.5) -- (0,3);
\end{tikzpicture} ; x, y \big) =~&  2 \big(  P(x \cdot P(y)) - P(x) P(y)  \big),  \\
     \llbracket P, P \rrbracket  \big(\begin{tikzpicture}[scale=0.10]
 \draw (0,0) -- (0,-2); \draw (0,0) -- (-3,3);  \draw (0,0) -- (3,3); \draw (-1.5,1.5) -- (0,3);
\end{tikzpicture}; x, y \big) =~&  2 \big(  P(P(x) \cdot y) - P(x) P(y)   \big), \\
      \llbracket P, P \rrbracket  \big(\begin{tikzpicture}[scale=0.10]
 \draw (0,0) -- (0,-2); \draw (0,0) -- (-3,3);  \draw (0,0) -- (3,3); \draw (0,0) -- (0,3);
\end{tikzpicture}; x, y \big) =~& 2 \big(  - P(x) P(y) \big).
\end{align*}
On the other hand, it follows from (\ref{dgla-d-formula}) that
\begin{align*}
   ( d P) \big(\begin{tikzpicture}[scale=0.10]
\draw (0,0) -- (0,-2); \draw (0,0) -- (-3,3);  \draw (0,0) -- (3,3); \draw (1.5,1.5) -- (0,3);
\end{tikzpicture} ; x, y \big) = (dP) \big(\begin{tikzpicture}[scale=0.10]
 \draw (0,0) -- (0,-2); \draw (0,0) -- (-3,3);  \draw (0,0) -- (3,3); \draw (-1.5,1.5) -- (0,3);
\end{tikzpicture}; x, y \big) = 0 \quad \text{ and } \quad (dP) \big(\begin{tikzpicture}[scale=0.10]
 \draw (0,0) -- (0,-2); \draw (0,0) -- (-3,3);  \draw (0,0) -- (3,3); \draw (0,0) -- (0,3);
\end{tikzpicture}; x, y \big) = \lambda P(xy).
\end{align*}
This shows that $dP + \frac{1}{2} \llbracket P, P \rrbracket = 0$ if and only if $P$ is a relative averaging operator of weight $\lambda$.
\end{proof}

Let $ P: B \rightarrow A$ be a relative averaging operator of weight $\lambda$ (i.e. $dP+\dfrac{1}{2}\llbracket P,P\rrbracket = 0$). We define a map $d_P:{C}^{n}( B,A)\rightarrow {C}^{n+1}( B,A)$ by $d_P(f) = df + \llbracket P,f\rrbracket$, for $f\in C^n(B,A)$.
Then we have the following.

\begin{prop}
The map $d_P$ is a differential, that is, $(d_P)^2 = 0$.
\end{prop}

\begin{proof}
 For any $f\in C^n(B,A)$, we have 
 \begin{align*}
 (d_P)^2(f) &= d_P\big(df +  \llbracket P,f\rrbracket\big) \\ 
 &= d^2 f + d  \llbracket P,f\rrbracket +  \llbracket P,df\rrbracket + \llbracket P, \llbracket P,f\rrbracket \rrbracket\\
 &=  \llbracket dP,f\rrbracket -  \llbracket P,df\rrbracket +  \llbracket P,df\rrbracket + \dfrac{1}{2}\llbracket \llbracket P, P \rrbracket, f \rrbracket\\
 &= \llbracket dP + \dfrac{1}{2} \llbracket P,P\rrbracket,f \rrbracket = 0.
 \end{align*} 
 This completes the proof.
\end{proof}

It follows from the above proposition that $\{C^\bullet(B, A),d_P\}$ is a cochain complex. The corresponding cohomology is called the {\bf cohomology} of the operator $P$. We denote the corresponding $n$-th cohomology group by $H_P^n(B, A)$.

Note that the differential $d_P$ is also a derivation for the graded Lie bracket $\llbracket ~, ~\rrbracket$. To see this, we observe that 
\begin{align*}
 d_P \llbracket f,g \rrbracket  &= d\llbracket f,g\rrbracket + \llbracket P, \llbracket f,g \rrbracket \rrbracket\\
 &= \llbracket df,g\rrbracket + (-1)^m \llbracket f,dg \rrbracket + \llbracket P, \llbracket f,g\rrbracket \rrbracket + (-1)^m \llbracket f, \llbracket P,g \rrbracket \rrbracket\\
 &= \llbracket d_P(f),g\rrbracket + (-1)^m \llbracket f,d_P(g)\rrbracket,
\end{align*}
for $f \in C^m(B,A)$ and $g\in C^n(B,A)$. In other words, the triple $(\displaystyle \oplus_{n=1}^{\infty} {C}^{n}( B,A),\llbracket ~,~\rrbracket,d_P )$ is a differential graded Lie algebra. This differential graded Lie algebra plays an important role in the deformations of the operator $P$. This is justified by the following result.

\begin{prop}
 Let $P: B\rightarrow A$ be a relative averaging algebra of weight $\lambda$. Given any other linear map $P' : B\rightarrow A$, the sum $P+P'$ is also a relative averaging algebra of weight $\lambda$ if and only if $P'$ is a Maurer-Cartan element of the differential graded Lie algebra  $(\displaystyle \oplus_{n=1}^{\infty} C^{n}( B,A),\llbracket ~,~\rrbracket,d_P )$.
\end{prop}

\begin{proof}
For any $P' : B \rightarrow A$, we have
\begin{align*}
    d (P + P') + \frac{1}{2} \llbracket P+P' , P+ P' \rrbracket =~& dP + dP' + \frac{1}{2} \big(  \llbracket P, P \rrbracket + 2 \llbracket P, P' \rrbracket + \llbracket P', P' \rrbracket  \big) \\
    =~& dP' + \llbracket P, P' \rrbracket + \frac{1}{2} \llbracket P' , P' \rrbracket   \qquad (\text{as } dP + \frac{1}{2} \llbracket P, P \rrbracket = 0) \\
    =~& d_P (P') + \frac{1}{2} \llbracket P', P' \rrbracket.
\end{align*}
Hence the result follows.
\end{proof}  

The above proposition shows that the differential graded Lie algebra $(\displaystyle \oplus_{n=1}^{\infty} {C}^{n}( B,A),\llbracket ~,~\rrbracket,d_P )$ controls the deformations of the operator $P$. For this reason, we call the differential graded Lie algebra $(\displaystyle \oplus_{n=1}^{\infty} {C}^{n}( B, A), \llbracket ~,~\rrbracket,d_P)$ as the controlling algebra for the deformations of $P$.

\section{Relative averaging algebras of any nonzero weight}\label{sec5}

In this section, we first consider relative averaging algebras of any nonzero weight $\lambda$ and find their functorial relations with triassociative algebras. Next, we construct a $L_\infty$-algebra whose Maurer-Cartan elements correspond to relative averaging algebras of weight $\lambda$. Then by a twisting procedure, we construct the controlling $L_\infty$-algebra of a given relative averaging algebra of nonzero weight.

%\subsection{Functorial relations between relative averaging algebras of any nonzero weight and triassociative algebras}

\begin{defn}
A {\bf relative averaging algebra of weight $\lambda$} is a triple $(A, B, P)$ in which $A$ is an associative algebra, $B$ is an associative $A$-bimodule and $P : B \rightarrow A$ is a relative averaging operator of weight $\lambda$.
\end{defn}

%For our convenience, we denote a relative averaging algebra $(A, B, P)$ of weight $\lambda$ as above by the notation $B \xrightarrow{P} A.$

We often denote an averaging algebra of weight $\lambda$ as above by the notation $(A_\mu, B_\nu^{l,r}, P)$ when the underlying structures are required to mention.

\begin{remark}\label{avg-ravg}
Let $(A, P)$ be an averaging algebra of weight $\lambda$. Then $(A, A, P)$ is a relative averaging algebra of weight $\lambda$, where $A$ is equipped with the adjoint $A$-bimodule structure.
\end{remark}

\begin{defn}
Let $(A,B,P)$ and $(A', B', P')$ be two relative averaging algebras of weight $\lambda$. A {\bf morphism} of relative averaging algebras of weight $\lambda$ from $(A,B,P)$ to $(A', B', P')$ is given by a pair $(\varphi, \psi)$ of algebra homomorphisms $\varphi : A \rightarrow A'$ and $\psi : B \rightarrow B'$ satisfying additionally
\begin{align*}
\psi (a \cdot x) = \varphi (a) \cdot' \psi(x), \quad \psi (x \cdot a) = \psi (x) \cdot' \varphi(a) ~~~ \text{ and } ~~~ \varphi \circ P = P' \circ \psi, \text{ for } a \in A, x \in B. 
\end{align*}
Here $\cdot'$ denotes both the left and right $A'$-actions on $B'$.
\end{defn}

The collection of all relative averaging algebras of weight $\lambda$ and morphisms between them forms a category, denoted by ${\bf rAvg_\lambda}$. If $(A,B,P)$ is a relative averaging algebra of any nonzero weight $\lambda$ then $(A, B, \frac{1}{\lambda}P)$ is a relative averaging algebra of weight $1$ and vice-versa. Therefore, for any $\lambda \neq 0$, the category ${\bf rAvg_\lambda}$ is equivalent to the category ${\bf rAvg_1}$ of relative averaging algebras of weight $1$.

\begin{prop}\label{ravg-ind-triass}
(i) Let $(A,B,P)$ be a relative averaging algebra of weight $\lambda$. Then the vector space $B$ with the operations
\begin{align}\label{induced-triass}
x \dashv_P y := x \cdot P(y), \quad x \vdash_P y := P(x) \cdot y ~~~ \text{ and } ~~~ x \perp_P y := \lambda x y, \text{ for } x, y \in B,
\end{align}
is a triassociative algebra. The triassociative algebra $(B, \dashv_P, \vdash_P, \perp_P)$ is said to be induced from the given relative averaging algebra of weight $\lambda$.

(ii) Let $(A,B,P)$ and $(A', B', P')$ be two relative averaging algebras of weight $\lambda$, and $(\varphi, \psi)$ be a morphism between them. Then the linear map $\psi : B \rightarrow B'$ is a morphism of induced triassociative algebras from $(B, \dashv_P, \vdash_P, \perp_P)$ to $(B', \dashv_{P'}, \vdash_{P'}, \perp_{P'})$.
\end{prop}

\begin{proof}
(i) Since $P$ is a relative averaging operator of weight $\lambda$, it follows that $Gr(P)$ is a subalgebra of the triassociative algebra  $A \oplus_\lambda B$ (cf. Proposition \ref{gr}). As $Gr(P)$ is isomorphic to $B$, the space $B$ inherits a triassociative algebra structure. The triassociative structure is precisely given by (\ref{induced-triass}).

(ii) For any $x, y \in B$, we have
\begin{align*}
    &\psi (x \dashv_P y ) = \psi (x \cdot P(y)) = \psi (x) \cdot' \varphi P(y) = \psi (x) \cdot' P' \psi(y) = \psi (x) \dashv_{P'} \psi (y), \\
    &\psi (x \vdash_P y) = \psi (P(x) \cdot y) = \varphi P(x) \cdot' \psi (y) = P' \psi (x) \cdot' \psi(y) = \psi(x) \vdash_{P'} \psi (y), \\
    &\psi (x \perp_P y) = \psi (\lambda xy) = \lambda \psi(x) \psi(y) = \psi (x) \perp_{P'} \psi(y).
\end{align*}
This completes the proof.
\end{proof}

The above proposition shows that there is a functor $\mathcal{F}$ from the category ${\bf rAvg_\lambda}$ (in particular from the category ${\bf rAvg_\lambda}$) to the category ${\bf Triass}$ of triassociative algebras. In the following, we construct a functor from the category $\bf{Triass}$ to the category $\bf{rAvg_1}$ of relative averaging algebras of weight 1.

We start with a triassociative algebra $(D,\dashv,\vdash,\perp)$. Let $I$ be the ideal of $D$ generated by the set $\big\{x\dashv y - x\vdash y ~\big|~ x,y\in D\big\}$. Then the quotient space $D/I$ has the obvious associative algebra structure with the multiplication \begin{align*}
 [x] [y]:= [x\dashv y] = [x\vdash y], \text{ for } [x],[y]\in D/I.   
\end{align*}
We denote this associative algebra simply by $D_{ass}$. On the other hand, the vector space $D$ has the associative product $\perp$. Denote this associative algebra $D_{\perp}$. Moreover, the associative algebra $D_{\perp}$ can be given an associative $D_{ass}$-bimodule structure with the left and right $D_{ass}$-actions given by \begin{center}
 $[x]\cdot y = x\vdash y$ and  $y\cdot [x] = y\dashv x$, for $[x]\in D_{ass},~ y\in D_{\perp}$.
\end{center} 
With these notations, the quotient map $q:D_{\perp}\rightarrow D_{ass}$, $q(x) = [x]$, for $x\in D_{\perp}$, is a relative averaging operator of weight 1. In other words, $( D_{ass}, D_\perp, q)$ is a relative averaging algebra of weight $1$. Moreover, the induced triassociative algebra structure on $D$ coincides with the given one.

Let $(D,\dashv,\vdash,\perp)$ and $(D',{\dashv}',{\vdash}',{\perp}')$ be two triassociative algebras and $\psi:D\rightarrow D'$ be a morphism between them. Then it can check that the pair 
\begin{align*}
 (\varphi^{\psi}, \psi) :  ( D_{ass}, D_\perp, q) \rightarrow  ( D'_{ass} , D'_{\perp'}, q')
\end{align*}
is a morphism of relative averaging algebras of weight $1$, where $\varphi^\psi ([x]) = [\psi (x)]$ for $[x] \in D_{ass}$. These constructions yield a functor $\mathcal{G}$ from the category ${\bf Triass}$ to the category ${\bf rAvg_1}$. Moreover, we have the following result whose proof is straightforward.

\begin{prop}
    The functor $\mathcal{G}$ is left adjoint to the functor $\mathcal{F}$. More precisely, for any triassociative algebra $(D, \dashv, \vdash, \perp)$ and an averaging algebra $(A, B, P)$ of weight $1$, there is a bijection
    \begin{align*}
        \mathrm{Hom}_{\bf Triass} \big(  (D, \dashv, \vdash, \perp), (B, \dashv_P, \vdash_P, \perp_P) \big) ~ \cong ~ \mathrm{Hom}_{\bf rAvg_1} \big(  (D_{ass}, D_\perp, q), (A, B, P) \big).
    \end{align*}
\end{prop}

\medskip

\begin{remark}
In \cite{jaa} the authors have considered the relative Rota-Baxter algebra of weight $\lambda$ as a generalization of a Rota-Baxter algebra. Let $A$ be an associative algebra and $B$ an associative $A$-bimodule. A linear map $R:B \rightarrow A$ is a relative Rota-Baxter operator of weight $\lambda$ if it satisfies
\begin{align}\label{rel-rb-w}
 R(x)R(y)= R(R(x) \cdot y + x \cdot R(y)+ \lambda xy), \text{ for } x, y \in B.
  \end{align}
  In this case, the triple $(A, B, R)$ is called a relative Rota-Baxter algebra of weight $\lambda$. In \cite{fard} Ebrahimi-Fard showed that a relative Rota-Baxter algebra of weight $\lambda$ induces a tridendriform algebra structure on $B$, given by
  \begin{align*}
  x \prec y = x \cdot R(y), \quad  x \succ y = R(x) \cdot y ~~~~ \text{ and } ~~~~ x \curlyvee y = \lambda xy, \text{ for } x, y \in B.
  \end{align*}
(See \cite{loday-ronco} for the definition of a tridendriform algebra).  Further, it follows from (\ref{rel-rb-w}) that a relative averaging algebra of weight $\lambda$ is a relative Rota-Baxter algebra of weight $ - \lambda$. Next, suppose we start with a triassociative algebra $(D, \dashv, \vdash, \perp)$. Then we have seen earlier that $(D_{ass}, D_\perp, q)$ is a relative averaging algebra of weight $1$. As a consequence, $(D_{ass}, D_\perp, q)$ is a relative Rota-Baxter algebra of weight $-1$. Hence the vector space $D$ carries a tridendriform algebra structure with the operations
  \begin{align*}
  x \prec y := x \dashv y, \quad x \succ y := x \vdash y ~~~ \text{ and } ~~~ x \curlyvee y := - x \perp y, \text{ for } x, y \in D.
  \end{align*}
  With this tridendriform algebra (induced from a triassociative algebra), the diagram given in (\ref{tri-tri-diag}) commutes.
\end{remark}

\subsection{Some background on homotopy Lie algebras and Maurer-Cartan characterization of relative averaging algebras of nonzero weight} 
In this subsection, we mainly recall some definitions regarding $L_\infty$-algebras (strongly homotopy Lie algebras). Throughout the present paper, we follow the definition of a $L_\infty$-algebra given in \cite{lada-markl}.
 \begin{defn}
 A {\bf $L_\infty$-algebra} is a pair $(\mathcal{L},\{l_k\}_{k=1}^{\infty})$ consisting of a graded vector space $\mathcal{L} = \displaystyle\oplus_{i\in \mathbb{Z}} \mathcal{L}_i$ equipped with a collection $\{l_k:{\mathcal{L}}^{\otimes k}\rightarrow \mathcal{L} \}_{k=1}^{\infty}$ of degree $1$ graded symmetric linear maps  satisfying the following higher Jacobi identities: 
 \begin{align}
\displaystyle\sum_{i+j=n+1}~\sum_{\sigma\in Sh(i,n-i)} l_j\big(l_i(x_{\sigma(1)},\ldots,x_{\sigma(i)}),x_{\sigma(i+1)},\ldots,x_{\sigma(n)}\big) = 0,
 \end{align}
 for all $n\geq 1$ and homogeneous elements $x_1,\ldots,x_n \in \mathcal{L}$.
 \end{defn}
 Any (differential) graded Lie algebra can be realized as a $L_{\infty}$-algebra by a degree shift. More precisely, let $(\mathfrak{g},[~,~],d)$ be a differential graded Lie algebra. Then $(s^{-1}\mathfrak{g},\{l_k\}_{k=1}^{\infty}))$ is a $L_{\infty}$-algebra, where \begin{align*}
 l_1(s^{-1}x) = s^{-1}(dx), ~~~ l_2(s^{-1}x,s^{-1}y) = (-1)^{\lvert x \rvert} s^{-1}[x,y]~~ \text{ and }~~l_k = 0 \text{ for } k\geq 3.   
 \end{align*}
A $L_\infty$-algebra  $(\mathcal{L},\{l_k\}_{k=1}^{\infty})$ is said to be weakly filtered if there exists a descending filtration $\mathcal{L}= {\mathcal{F}}_1\mathcal{L}\supset {\mathcal{F}}_2\mathcal{L} \supset \cdots  \supset {\mathcal{F}}_n\mathcal{L}\supset \cdots $ and a natural number $N\in \mathbb{N}$ such that 
\begin{align*}
     \mathcal{L} = \lim {\mathcal{L}/{\mathcal{F}}_n\mathcal{L}} ~~ \text{ and } ~~ l_k(x_1,\ldots,x_k) \in {\mathcal{F}}_k\mathcal{L}, \text{ for all } k\geq N \text{ and } x_1,\ldots,x_n \in \mathcal{L}.
    \end{align*}
Weakly filtered $L_\infty$-algebras are useful to make sense of certain infinite sums in $\mathcal{L}$ \cite{getzler}. Throughout our paper, we assume that all $L_\infty$-algebras are weakly filtered.
\begin{defn}
  Let $(\mathcal{L},\{l_k\}_{k=1}^{\infty})$ be a $L_{\infty}$-algebra. An element $\alpha\in \mathcal{L}_0$ is said to be a {\bf Maurer-Cartan element} of the $L_{\infty}$-algebra if \begin{align*}
  \displaystyle\sum_{k=1}^{\infty} \dfrac{1}{k!}~l_k(\alpha,\ldots,\alpha) = 0.   
  \end{align*} 
\end{defn}
If $\alpha$ is a Maurer-Cartan element of the $L_{\infty}$-algebra $(\mathcal{L},\{l_k\}_{k=1}^{\infty})$ then one can construct a new $L_{\infty}$-algebra $(\mathcal{L},\{{l^\alpha_k}\}_{k=1}^{\infty})$ on the same graded vector space $\mathcal{L}$. The structure maps $\{{l^\alpha_k}\}_{k=1}^{\infty}$ for this new $L_{\infty}$-algebra are given by \begin{align*}
 {l^\alpha_k}(x_1,\ldots,x_k) =  \displaystyle\sum_{n=0}^{\infty} \dfrac{1}{n!}~l_{n+k}(\alpha,\ldots,\alpha,x_1,\ldots,x_k), \text{ for } x_1,\ldots x_k\in \mathcal{L}.   
\end{align*}

The $L_{\infty}$-algebra $(\mathcal{L},\{{l^\alpha_k}\}_{k=1}^{\infty})$ is said to be obtained from $(\mathcal{L},\{l_k\}_{k=1}^{\infty})$ twisted by the Maurer-Cartan element $\alpha$. For any $\alpha'\in {\mathcal{L}}_0$, the sum $\alpha + {\alpha}'$ is a Maurer-Cartan element of the $L_{\infty}$-algebra  $(\mathcal{L},\{l_k\}_{k=1}^{\infty})$ if and only if $\alpha'$ is a Maurer-Cartan element of the $L_{\infty}$-algebra $(\mathcal{L},\{{l^\alpha_k}\}_{k=1}^{\infty})$.

Next, we recall Voronov's construction of a $L_{\infty}$-algebra \cite{voro}. Let $(\mathfrak{g},\mathfrak{a},p,\Delta)$ be a quadruple consisting of a graded Lie algebra $\mathfrak{g}$ (with the bracket [~,~]), an abelian Lie subalgebra $\mathfrak{a}\subset \mathfrak{g}$, a projection map $p : \mathfrak{g}\rightarrow \mathfrak{g}$ such that $\mathrm{ker}(p)\subset \mathfrak{g}$ is a graded Lie subalgebra and $\mathrm{im}(p) = \mathfrak{a}$, and an element $\Delta\in \mathrm{ker}{(p)}_1$ that satisfies $[\Delta,\Delta] =0$. Such a quadruple is called a {\em $V$-data}.

\begin{thm}\label{voro-thm}
  Let $(\mathfrak{g},\mathfrak{a},p,\Delta)$ be a $V$-data.
  
  (i) Then the graded vector space $\mathfrak{a}$ carries a $L_{\infty}$-algebra with the structure maps $\{l_k:{\mathfrak{a}}^{\otimes k}\rightarrow \mathfrak{a}\}_{k=1}^{\infty}$  given by \begin{align*}
  l_k(a_1,\ldots,a_k) = p[\cdots[[\Delta,a_1],a_2],\ldots,a_k],~~for~~k\geq 1.    
  \end{align*}
  
(ii) Let $\mathfrak{h}\subset \mathfrak{g}$ be a graded Lie subalgebra that satisfies $[\Delta,\mathfrak{h}]\subset \mathfrak{h}$. Then the graded vector space $s^{-1}\mathfrak{h}\oplus \mathfrak{a}$ carries a $L_{\infty}$-algebra with the structure maps \begin{align*}
l_1\big((s^{-1}x,a)\big) =~& \big(-s^{-1}[\Delta,x],~ p(x+[\Delta,a])\big),\\ 
l_2\big((s^{-1}x,0),(s^{-1}y,0)\big) =~& \big((-1)^{ \lvert x \rvert} s^{-1}[x,y],0\big),\\
l_k\big((s^{-1}x,0),(0,a_1),\ldots,(0,a_{k-1})\big) =~& \big(0,p~[\cdots [[x,a_1],a_2],\ldots,a_{k-1} ]\big),~~ for~ k\geq 2,\\
l_k\big((0,a_1),\ldots,(0,a_{k})\big) =~& \big(0,p~[\cdots [[\Delta,a_1],a_2 ],\ldots,a_{k}]\big),~~ for~ k\geq 2. 
\end{align*} 
Here $x,y$ are homogeneous elements of $\mathfrak{h}$ and $a,a_1,\ldots,a_k$ are homogeneous elements of $\mathfrak{a}$. Up to the permutation of the above inputs, all other maps vanish. 
\end{thm}

%\subsection{Maurer-Cartan characterization of relative averaging algebras of any nonzero weight}

\medskip

Let $A$ and $B$ be two vector spaces. Consider the graded Lie algebra
\begin{align*}
    \mathfrak{g} = \big(  \oplus_{n=0}^\infty C^{n+1} (A \oplus B, A \oplus B) = \oplus_{n=0}^\infty \mathrm{Hom} ( \mathbb{K}[T_{n+1}] \otimes (A \oplus B)^{\otimes n+1}, A \oplus B   ) , [~,~]  \big)
\end{align*}
associated with the direct sum vector space $A \oplus B$. Then we have the following.

\begin{prop}
    The graded subspace $\mathfrak{a} = \oplus_{n=0}^\infty C^{n+1} (B, A) = \oplus_{n=0}^\infty \mathrm{Hom} (\mathbb{K} [T_{n+1}] \otimes B^{\otimes n}, A)$ is an abelian Lie subalgebra of $\mathfrak{g}$.
\end{prop}

\begin{proof}
Let $f \in C^m (B, A)$ and $g \in C^n (B, A)$. Then it follows from (\ref{tri-circ}) that $f \circ_i g = 0$, for all $1 \leq i \leq m$. Hence $[f,g] = 0$. This completes the proof.
\end{proof}

Let $\mathcal{A}^{k,l}$ be the direct sum of all possible $(k+l)$ tensor powers of $A$ and $B$ in which $A$ appears $k$ times and $B$ appears $l$ times. That is,
\begin{align*}
    \mathcal{A}^{1,1} = (A \otimes B) \oplus (B \otimes A), \quad  \mathcal{A}^{2,1} = (A \otimes A \otimes B) \oplus (A \otimes B \otimes A) \oplus (B \otimes A \otimes A) ~~~~ \text{etc.}
\end{align*}
Let $p: \mathfrak{g} \rightarrow \mathfrak{g}$ be the projection onto the subspace $\mathfrak{a}$. Then 
\begin{align*}
    \mathrm{ker}(p) =~& \bigoplus _{n=0}^\infty \bigg(    \oplus_{ \substack{k+l=n+1\\ k > 0 , l \geq 0}} \mathrm{Hom} (   \mathbb{K} [T_{n+1}] \otimes {\mathcal{A}^{k, l}}, A) \oplus \mathrm{Hom} ( \mathbb{K}[T_{n+1}] \otimes (A \oplus B)^{\otimes n+1}, B)  \bigg) \\
    =~& \bigoplus_{n=0}^\infty \bigg(    \oplus_{ \substack{k+l=n+1 \\ k > 0, l \geq 0}} \mathrm{Hom} (   \mathbb{K} [T_{n+1}] \otimes {\mathcal{A}^{k, l}}, A) \oplus     \oplus_{\substack{k+l=n+1 \\ k , l \geq 0}}  \mathrm{Hom} ( \mathbb{K}[T_{n+1}] \otimes \mathcal{A}^{k,l},  B)  \bigg).
\end{align*}
It is not hard to see that $\mathrm{ker}(p) \subset \mathfrak{g}$ is a graded Lie subalgebra. Hence $(\mathfrak{g}, \mathfrak{a}, p, \Delta = 0)$ is a $V$-data.

Let $\mathfrak{h} \subset \mathfrak{g}$ be the subspace given by
\begin{align*}
     \mathfrak{h} =~& \bigoplus_{n=0}^\infty \bigg(    \mathrm{Hom} (   \mathbb{K} [T_{n+1}] \otimes A^{\otimes n+1}, A)    \oplus_{ \substack{k+l=n+1\\ k \geq 0, l > 0}}  \mathrm{Hom} ( \mathbb{K} [T_{n+1}] \otimes \mathcal{A}^{k,l},  B)  \bigg) \\
     =~& \bigoplus_{n=0}^\infty \bigg(    \mathrm{Hom} (  \mathbb{K} [T_{n+1}] \otimes A^{\otimes n+1}, A)  \oplus  \mathrm{Hom} (  \mathbb{K} [T_{n+1}] \otimes B^{\otimes n+1}, B   )   \oplus_{ \substack{k+l=n+1 \\ k, l > 0}}  \mathrm{Hom} ( \mathbb{K} [T_{n+1}] \otimes \mathcal{A}^{k,l},  B)  \bigg).
\end{align*}
Then it can be checked that the subspace $\mathfrak{h} \subset \mathfrak{g}$ is a graded Lie subalgebra. Thus, it follows from Theorem \ref{voro-thm} that the graded space $s^{-1} \mathfrak{h} \oplus \mathfrak{a}$ carries a $L_\infty$-algebra structure with the operations
\begin{align*}
    l_2 \big(  (s^{-1}f, 0), (s^{-1} g,  0) \big) =~& ( (-1)^{|f|} s^{-1}[f,g], 0),\\
    l_k \big(  (s^{-1} f, 0), (0, h_1), \ldots, (0, h_{k-1})   \big) =~& (0, p [ \cdots [[f, h_1], h_2], \ldots, h_{k-1}]), \text{ for } k \geq 2
\end{align*}
and up to permutations of the above inputs, all other maps vanish.
Note that the graded vector space $s^{-1} \mathfrak{h} \oplus \mathfrak{a}$ is componentwise given by
\begin{align*}
(s^{-1} \mathfrak{h} \oplus \mathfrak{a})_{-1} =~& \mathfrak{h}_0 = \mathrm{Hom}(A,A) \oplus \mathrm{Hom}(B,B) \quad  (\text{we assume that } \mathfrak{a}_{-1} = 0) \\
(s^{-1} \mathfrak{h} \oplus \mathfrak{a})_{n \geq 0} =~& \mathrm{Hom}( \mathbb{K}[T_{n+2}] \otimes A^{\otimes n+2}, A) \oplus \mathrm{Hom}( \mathbb{K}[T_{n+2}] \otimes B^{\otimes n+2}, B) \\
 ~& \qquad \oplus_{ \substack{k+l=n+2 \\ k, l > 0}}  \mathrm{Hom} ( \mathbb{K} [T_{n+2}] \otimes \mathcal{A}^{k,l},  B) \oplus \mathrm{Hom} (\mathbb{K}[T_{n+1}] \otimes B^{\otimes n+1}, A).
\end{align*}

Next, suppose that the vector spaces $A$ and $B$ are equipped with the maps
\begin{align*}
\mu : A \otimes A \rightarrow A, \quad \nu : B \otimes B \rightarrow B, \quad l : A \otimes B \rightarrow B, \quad r : B \otimes A \rightarrow B ~~~~ \text{ and } ~~~~ P: B \rightarrow A.
\end{align*}
We define an element $\pi_\lambda \in \mathfrak{h}_1$ by
\begin{align*}
    \pi_\lambda  \big(  \begin{tikzpicture}[scale=0.10]
\draw (0,0) -- (0,-2); \draw (0,0) -- (-3,3);  \draw (0,0) -- (3,3); \draw (1.5,1.5) -- (0,3);
\end{tikzpicture} ;(a,x),(b,y) \big) &= ( \mu (a, b), r(x , b) ), \quad \pi_\lambda \big(\begin{tikzpicture}[scale=0.10]
 \draw (0,0) -- (0,-2); \draw (0,0) -- (-3,3);  \draw (0,0) -- (3,3); \draw (-1.5,1.5) -- (0,3);
\end{tikzpicture};(a,x),(b,y) \big) = ( \mu (a, b) , l (a, y)) \\
& \text{ and } ~~  \pi_\lambda \big(\begin{tikzpicture}[scale=0.10]
 \draw (0,0) -- (0,-2); \draw (0,0) -- (-3,3);  \draw (0,0) -- (3,3); \draw (0,0) -- (0,3);
\end{tikzpicture};(a,x), (b,y) \big) = ( \mu (a, b), \lambda \nu (x, y) ), 
\end{align*}
for $(a,x), (b, y) \in A \oplus B$. Then we have the following.

%The importance of the above $L_\infty$-algebra is given by the following result.

\begin{thm}
With the above notations, $A_\mu$ is an associative algebra, $B_\nu^{l, r}$ is an associative $A_\mu$-bimodule and $P : B \rightarrow A$ is a relative averaging operator of nonzero weight $\lambda$ if and only if $(s^{-1} {\pi_\lambda}, P)$ is a Maurer-Cartan element of the $L_\infty$-algebra $(s^{-1} \mathfrak{h} \oplus \mathfrak{a}, \{ l_k \}_{k=1}^\infty)$ constructed above. In other words,
\begin{align*}
&(A_\mu, B_\nu^{l, r} , P) \text{ is a relative averaging algebra of weight } \lambda ~  \\ & \quad \Leftrightarrow ~ (s^{-1} {\pi_\lambda}, P) \text{ is a Maurer-Cartan element of } (s^{-1} \mathfrak{h} \oplus \mathfrak{a}, \{ l_k \}_{k=1}^\infty).
\end{align*}
\end{thm}

\begin{proof}
Note that
\begin{align*}
l_2 \big( (s^{-1} \pi_\lambda, P), (s^{-1} \pi_\lambda, P)      \big) = \big( - s^{-1} [\pi_\lambda, \pi_\lambda] , 2 l_2 (  (s^{-1} \pi_\lambda, 0), (0,P)    )   \big) = (-s^{-1} [\pi_\lambda, \pi_\lambda], 2 [\pi_\lambda, P]).
\end{align*}
On the other hand, it is straightforward to observe that
$[[[ \pi_\lambda, P], P], P] = 0$. As a consequence,
\begin{align*}
l_k \big(  (s^{-1} \pi_\lambda, P), \ldots, (s^{-1} \pi_\lambda, P)   \big) = 0, \text{ for } k \geq 4.
\end{align*}
Hence
\begin{align*}
&\sum_{k=1}^\infty \frac{1}{k!} l_k \big(  (s^{-1} \pi_\lambda, P), \ldots, (s^{-1} \pi_\lambda, P)   \big) \\
&= \frac{1}{2} l_2 \big( (s^{-1} \pi_\lambda, P), (s^{-1} \pi_\lambda, P)   \big) + \frac{1}{6} l_3 \big(  (s^{-1} \pi_\lambda, P), (s^{-1} \pi_\lambda, P), (s^{-1} \pi_\lambda, P)  \big) \\
&= \frac{1}{2} \big( -s^{-1} [\pi_\lambda, \pi_\lambda], 2 [\pi_\lambda, P]   \big) + \frac{1}{6} \big( 0, 3 [[ \pi_\lambda, P], P]   \big) \\
&= \big(  - \frac{1}{2} s^{-1} [\pi_\lambda, \pi_\lambda], [\pi_\lambda, P] + \frac{1}{2} [[\pi_\lambda, P], P]   \big).
\end{align*}
This shows that $(s^{-1} \pi_\lambda , P) \in (s^{-1} \mathfrak{h} \oplus \mathfrak{a})_0$ is a Maurer-Cartan element of the $L_\infty$-algebra if and only if $[\pi_\lambda, \pi_\lambda] = 0$ (i.e. $A_\mu$ is an associative algebra and $B_\nu^{l, r}$ is an associative $A_\mu$-bimodule) and 
\begin{align*}
[\pi_\lambda , P] + \frac{1}{2} [[\pi_\lambda, P], P] = 0
\end{align*}
(i.e. $P$ is a relative averaging operator of weight $\lambda$ (cf. Theorem \ref{dgla-thmm})). This completes the proof.
\end{proof}

Let $(A_\mu, B_\nu^{l, r} ,{P})$ be a relative averaging algebra of weight $\lambda$ with  $\alpha = (s^{-1} {\pi_\lambda}, P)$ being the corresponding Maurer-Cartan element of the $L_\infty$-algebra $(s^{-1} \mathfrak{h} \oplus \mathfrak{a}, \{ l_k \}_{k=1}^\infty)$. Hence one can construct the new $L_\infty$-algebra $(s^{-1} \mathfrak{h} \oplus \mathfrak{a}, \{ l_k^{  (s^{-1} {\pi_\lambda}, P)   } \}_{k=1}^\infty)$ twisted by the Maurer-Cartan element $\alpha = (s^{-1} {\pi_\lambda}, P)$.

\begin{thm}\label{sum-l-inf}
Let $(A_\mu, B_\nu^{l, r} , {P} )$ be a relative averaging algebra of weight $\lambda$. Then for any linear maps
\begin{align*}
\mu' : A \otimes A \rightarrow A, \quad \nu' : B \otimes B \rightarrow B, \quad l': A \otimes B \rightarrow B, \quad r' : B \otimes A \rightarrow B ~~~~ \text{ and } ~~~~ P' : B \rightarrow A,
\end{align*}
the structure $(A_{\mu + \mu'}, B_{\nu + \nu'}^{l+l', r + r'}, {P + P'})$ is a relative averaging algebra of weight $\lambda$ if and only if $(s^{-1} \pi_\lambda', P') \in (s^{-1} \mathfrak{h} \oplus \mathfrak{a})_0$ is a Maurer-Cartan element of the twisted $L_\infty$-algebra $(s^{-1} \mathfrak{h} \oplus \mathfrak{a}, \{ l_k^{  (s^{-1} {\pi_\lambda}, P)   } \}_{k=1}^\infty)$.
\end{thm}

The above proposition says that the $L_\infty$-algebra $(s^{-1} \mathfrak{h} \oplus \mathfrak{a}, \{ l_k^{  (s^{-1} {\pi_\lambda}, P)  } \}_{k=1}^\infty)$ controlls the Maurer-Cartan deformations of the given relative averaging algebra $(A_\mu, B_\nu^{l, r}, {P})$ of weight $\lambda$. For this reason, we call the $L_\infty$-algebra $(s^{-1} \mathfrak{h} \oplus \mathfrak{a}, \{ l_k^{  (s^{-1} {\pi_\lambda}, P)   } \}_{k=1}^\infty)$ as the controlling algebra {for the deformations of} $(A_\mu, B_\nu^{l, r}, P)$.

%\section{Cohomology of averaging algebras}

\section{Cohomology of (relative) averaging algebras of nonzero weight}\label{sec6}
In this section, we define the cohomology of a relative averaging algebra of nonzero weight $\lambda$. In particular, we describe the cohomology of an averaging algebra of weight $\lambda$. Finally, given a relative averaging algebra $(A, B, {P})$ of weight $\lambda$, we find a long exact sequence connecting the cohomology groups of the operator $P$ and the cohomology groups of the full relative averaging algebra $(A, B, {P})$ of weight $\lambda$.

Let $(A, B, {P})$ be a relative averaging algebra of weight $\lambda$ with the corresponding Maurer-Cartan element $(s^{-1} \pi_\lambda, P)$. For each $n \geq 0$, we define the space  $C^n_{\mathrm{rAvg}_\lambda} (A, B , P)$ of $n$-cochains by $C^0_{\mathrm{rAvg}_\lambda} (A, B , P) = 0$, $C^1_{\mathrm{rAvg}_\lambda} (A, B , P) = \mathrm{Hom}(A, A) \oplus \mathrm{Hom}(B,B)$ and 

\begin{align*}
C^{n \geq 2}_{\mathrm{rAvg}_\lambda} (A, B, P) = 
\mathrm{Hom}(A^{\otimes n}, A) \oplus \mathrm{Hom}(B^{\otimes n}, B)   \oplus_{ \substack{k+l = n+1\\ k, l \geq 1}} \mathrm{Hom}(\mathcal{A}^{k, l}, B)  \oplus \mathrm{Hom} ( \mathbb{K} [T_{n-1}] \otimes B^{\otimes n-1}, A).
\end{align*}

\medskip

\begin{remark}
(i) Note that an element $(f, g) \in C^1_{\mathrm{rAvg}_\lambda} (A, B, P)$ corresponds to the element $(s^{-1} (f+g), 0) \in (s^{-1} \mathfrak{h} \oplus \mathfrak{a})_{-1}$.

\medskip

%(ii) For any $n \geq 1$, there is an embedding
%\begin{align}\label{embedd}
%\mathrm{Hom} ( (A \oplus B)^{\otimes n}, A \oplus B) \hookrightarrow \mathrm{Hom} ( \mathbb{K}[T_n] \otimes (A \oplus B)^{\otimes n}, A \oplus B), ~ f \mapsto \widetilde{f}
%\end{align}
%given by $\widetilde{f} \big( T; (a_1, x_1), \ldots, (a_n, x_n)    \big) = f \big( (a_1, x_1), \ldots, (a_n, x_n) \big)$, for any $T \in T_n$ and $(a_1, x_1), \ldots, (a_n, x_n) \in A \oplus B$. For any $f \in \mathrm{Hom} (  (A \oplus B)^{\otimes m}, A \oplus B)$ and $g \in \mathrm{Hom} (  (A \oplus B)^{\otimes n}, A \oplus B)$, we also have $\widetilde{  [f, g]_\mathsf{G}} = [\widetilde{f}, \widetilde{g}]$, where $[~,~]_\mathsf{G}$ is the classical Gerstenhaber bracket on the graded space $\oplus_{n=1}^\infty \mathrm{Hom} ( (A \oplus B)^{\otimes n}, A \oplus B)$. It follows from the embedding (\ref{embedd}) 

(ii) An element $(f,g,h, \gamma) \in C^{n \geq 2}_{\mathrm{rAvg}_\lambda} (A, B, P)$ gives rise to an element $(s^{-1} ( \pi_{f,g, h} ), \gamma) \in (s^{-1} \mathfrak{h} \oplus \mathfrak{a})_{n-2}$, where $\pi_{f,g,h} \in \mathfrak{h}_{n-1}$ is given by
\begin{align*}
    (\pi_{f,g,h}) (T; (a_1, x_1), \ldots, (a_n, x_n)) = \big(  f(a_1, \ldots, a_n), \lambda^{k-2} (g+h) (a_1, \ldots, a_{i_1 -1}, x_{i_1}, \ldots, x_{i_1 + \cdots + i_{k-1}}, \ldots, a_n)  \big),
\end{align*}
when $T= T^{i_1 -1} \vee \cdots \vee T^{i_k -1}$ for some unique $(i_1 -1)$-tree $T^{i_1 -1} \in T_{i_1 -1}, \ldots, (i_k -1)$-tree $T^{i_k -1} \in T_{i_k -1}$. This is in fact an embedding of the space $C^{n \geq 2}_{\mathrm{rAvg}_\lambda} (A,B,P)$ inside $(s^{-1} \mathfrak{g} \oplus \mathfrak{a})_{n-2}.$
\end{remark}

Using the above embeddings, we define a map $\delta_{\mathrm{rAvg}_\lambda} : C^n_{\mathrm{rAvg}_\lambda} (A, B, P) \rightarrow C^{n+1}_{\mathrm{rAvg}_\lambda} (A, B, P)$ by

\begin{align*}
\delta_{\mathrm{rAvg}_\lambda} (f,g) =~& - l_1^{  (s^{-1} \pi_\lambda, P)} ( s^{-1} ({f} + {g}), 0), \text{ for } (f,g) \in C^1_{\mathrm{rAvg}_\lambda} (A, B, P), \\
\delta_{\mathrm{rAvg}_\lambda} (f,g, h, \gamma) =~& (-1)^{n-2} ~ l_1^{  (s^{-1} \pi_\lambda, P)  } (  s^{-1}  ( \pi_{f,g,h} ), \gamma), \text{ for } (f,g, h, \gamma) \in C^{n \geq 2}_{\mathrm{rAvg}_\lambda} (A, B, P).
\end{align*}

\begin{prop}
With the above notations, we have $(\delta_{\mathrm{rAvg}_\lambda})^2 = 0$. That is, $\{  C^\bullet_{\mathrm{rAvg}_\lambda} (A, B, P) , \delta_{\mathrm{rAvg}_\lambda} \}$ is a cochain complex.
\end{prop}

\begin{proof}
Since the controlling algebra $(s^{-1} \mathfrak{h} \oplus \mathfrak{a} ,    \{ l_k^{(s^{-1} \pi_\lambda, P)}   \}_{k=1}^\infty)$ is a $L_\infty$-algebra, it follows that $(l_1^{(s^{-1} \pi_\lambda, P)})^2 = 0$.
\end{proof}

Let $(A, B, P)$ be a relative averaging algebra of weight $\lambda$. Then the cohomology of the cochain complex  $\{  C^\bullet_{\mathrm{rAvg}_\lambda} (A, B, P) , \delta_{\mathrm{rAvg}_\lambda} \}$ is called the {\bf cohomology} of $(A, B, P)$. We denote the corresponding $n$-th cohomology group by $H^n_{\mathrm{rAvg}_\lambda} (A, B, P)$.

Note that the map $\delta_{\mathrm{rAvg}_\lambda}$ can be described as
\begin{align*}
&\delta_{\mathrm{rAvg}_\lambda} (f,g,h, \gamma) \\
&= (-1)^{n-2}~ l_1^{ (s^{-1} \pi_\lambda , P)} \big(  s^{-1} (\pi_{f,g,h}) ,  \gamma  \big) \\
&= (-1)^{n-2} \sum_{k=0}^\infty \frac{1}{k!} l_{k+1} \big(  \underbrace{     (s^{-1} \pi_\lambda , P), \ldots, (s^{-1} \pi_\lambda , P) }_{k ~ \mathrm{ many}},  (  s^{-1} (\pi_{f,g,h}) ,  \gamma )     \big) \\
&= (-1)^{n-2} \bigg\{ l_2 \big(  (s^{-1} \pi_\lambda, 0), (s^{-1} (  \pi_{f,g,h} ), 0)  \big)  + l_2 \big( (s^{-1} \pi_\lambda, 0), (0, \gamma)  \big) + l_3 \big( (s^{-1} \pi_\lambda, 0), (0,P), (0, \gamma)   \big) \\
&\qquad \qquad + \frac{1}{n!} l_{n+1} \big(  ( s^{-1}   (\pi_{f,g,h}), 0 ),  (0,P), \ldots, (0, P)   \big)  \bigg\} \quad (\text{all other terms are zero}) \\
&= \big(  (-1)^{n-1} s^{-1} [\pi_\lambda, \pi_{f,g,h}] , ~\underbrace{(-1)^{n} [\pi_\lambda, \gamma] + (-1)^n [[\pi_\lambda, P], \gamma]}_{ = (-1)^{n-1} d_P (\gamma) } \\
& \qquad \qquad \qquad \qquad \qquad \qquad \qquad \qquad + \frac{(-1)^n}{n!} [ \cdots [[ \pi_{f,g,h}, P], P], \ldots, P]  \big).
\end{align*}
Thus, using the above identifications, the coboundary map $\delta_{\mathrm{rAvg}_\lambda}$ can be explicitly given by
\begin{align*}
\delta_{\mathrm{rAvg}_\lambda} (f,g,h, \gamma) = \big( \delta_\mathrm{Hoch}^A (f) , \delta^B_\mathrm{Hoch} (g) , \delta_\mathrm{Hoch}^{f,g} (h) , (-1)^{n-1} d_P (\gamma) + \theta_P (f,g,h)   \big).
\end{align*}
Here $\delta_\mathrm{Hoch}^A$  (resp. $\delta_\mathrm{Hoch}^B$) is the Hochschild coboundary map of the associative algebra $A$ (resp. $B$). For any $f \in \mathrm{Hom}(A^{\otimes n}, A)$ and $g \in \mathrm{Hom}(B^{\otimes n}, B)$, the map 
\begin{align*}
\delta^{f,g}_\mathrm{Hoch} : \oplus_{\substack{ k+l=n \\ k, l \geq 1}} \mathrm{Hom}(\mathcal{A}^{k,l}, B) \rightarrow \oplus_{\substack{ k+l=n+1 \\ k, l \geq 1}} \mathrm{Hom}(\mathcal{A}^{k,l}, B) ~ \text{ is given by }
\end{align*}
\begin{align*}
\big( \delta_\mathrm{Hoch}^{f,g} (h) \big) (a_1, \ldots, a_{n+1}) =~& (\nu + l+r) (a_1, (f+g+ h)(a_2, \ldots, a_{n+1})) \\
&+ \sum_{i=1}^n (-1)^i (h +\nu) \big( a_1, \ldots, a_{i-1} , (\mu + \nu + l + r) (a_i, a_{i+1}), \ldots, a_{n+1}  \big) \\
&+ (-1)^{n+1} (\nu + l+ r) ( (f+g+h)(a_1, \ldots, a_n) , a_{n+1}),
\end{align*}
for $a_1 \otimes \cdots \otimes a_{n+1} \in \mathcal{A}^{k,l}$ with $k+l = n+1$ and $k, l \geq 1$. The map $d_P$ is the coboundary operator of the relative averaging operator $P$ of weight $\lambda$. Finally, for any $T \in T_n$ (say $T= T^{i_1 -1} \vee \cdots \vee T^{i_k -1}$ for some unique $(i_1 -1)$-tree $T^{i_1 -1} \in T_{i_1 -1}, \ldots, (i_k -1)$-tree $T^{i_k -1} \in T_{i_k -1}$), the map $\theta_P (f,g,h) \in \mathrm{Hom}(\mathbb{K}[T_n]\otimes B^n, A)$ is given by
\begin{align*}
&(\theta_P (f,g,h)) (T; x_1, \ldots, x_n) 
= (-1)^n \big\{ f\big( P(x_1), \ldots, P(x_n) \big) \\ & \qquad - \lambda^{k-2} P(g+ h) \big(P(x_1), \ldots, P(x_{i_1 -1}),  x_{i_1}, P(x_{i_1 + 1}),\ldots, x_{i_1 + \cdots + i_{k-1}} , \ldots , P(x_n)  \big)   \big\}.
\end{align*}

%\textcolor{red}{Write down the explicit} $\delta_{\mathrm{rAvg}_\lambda}$.

\medskip

Let $A$ be an associative algebra and $B$ be an associative $A$-bimodule. Then one may consider the cochain complex $\{ C^\bullet_\mathrm{AssAct} (A, B), \delta_\mathrm{AssAct} \}$, where $C^0_\mathrm{AssAct} (A, B) = 0$, $C^1_\mathrm{AssAct} (A, B) = \mathrm{Hom}(A,A) 
 \oplus \mathrm{Hom} (B, B)$ and
\begin{align*}
C^{n \geq 2}_\mathrm{AssAct} (A, B) = \mathrm{Hom}(A^{\otimes n}, A) \oplus \mathrm{Hom}(B^{\otimes n}, B)   \oplus_{ \substack{k+l = n+1\\ k, l \geq 1}} \mathrm{Hom}(\mathcal{A}^{k, l}, B).
\end{align*}
The coboundary map $\delta_\mathrm{AssAct} : C^n_\mathrm{AssAct} (A, B) \rightarrow C^{n+1}_\mathrm{AssAct} (A, B) $ is given by
\begin{align*}
\delta_\mathrm{AssAct} (f,g, h) = (\delta_\mathrm{Hoch}^A (f) , \delta_\mathrm{Hoch}^B (g), \delta_\mathrm{Hoch}^{A, B} (h)).
\end{align*}
We denote the $n$-th cohomology of the cochain complex $\{ C^\bullet_\mathrm{AssAct} (A, B), \delta_\mathrm{AssAct} \}$ by $H^n_\mathrm{AssAct} (A, B)$.

\begin{thm}
Let $(A, B, P)$ be a relative averaging algebra of weight $\lambda$. Then there is a long exact sequence
\begin{align}\label{long-exc}
\cdots \longrightarrow H^{n-1}_P (B, A) \longrightarrow H^n_\mathrm{ rAvg_\lambda} (A, B, P) \longrightarrow H^n_\mathrm{AssAct} (A, B) \longrightarrow H^n_P (B,A) \longrightarrow \cdots.
\end{align}
\end{thm}

\begin{proof}
There is an obvious short exact sequence of cochain complexes
\begin{align*}
     0 \longrightarrow \{ C^{\bullet -1} (B, A), d_P \} \longrightarrow \{ C^\bullet_\mathrm{ rAvg_\lambda} (A, B, P) , \delta_\mathrm{rAvg_\lambda} \} \longrightarrow \{     C^\bullet_\mathrm{AssAct} (A, B), \delta_\mathrm{AssAct}   \} \longrightarrow 0.
\end{align*}
This yields the long exact sequence (\ref{long-exc}) on the level of cohomology.
\end{proof}

\subsection{Cohomology of an averaging algebra of any nonzero weight}

In this subsection, we introduce the cohomology of an averaging algebra of weight $\lambda$ using the general theory of relative averaging algebras. Subsequently, we also find a long exact sequence connecting cohomology groups.

Let $(A, P)$ be an averaging algebra of nonzero weight $\lambda$. We define $C^0_{\mathrm{Avg}_\lambda} (A,P) = 0$, $C^1_{\mathrm{Avg}_\lambda} (A,P) = \mathrm{Hom}(A,A)$ and 
\begin{align*}
C^{n \geq 2}_{\mathrm{Avg}_\lambda} (A,P) = \mathrm{Hom}(A^{\otimes n}, A) \oplus \mathrm{Hom}( \mathbb{K}[T_{n-1}] \otimes A^{\otimes n-1}, A).
\end{align*}
Note that the averaging algebra $(A, P)$ of weight $\lambda$ can be seen as a relative averaging algebra $(A, A , {P})$ of weight $\lambda$ (see Remark \ref{avg-ravg}). Then for each $n \geq 0$, there is an embedding
\begin{align*}
i: C^n_{\mathrm{Avg}_\lambda} (A, P) \hookrightarrow C^n_{\mathrm{rAvg}_\lambda} (A, A, P), ~ (f, \gamma) \mapsto (f, f, f, \gamma).
\end{align*}
Further, it can be easily checked that $\delta_{\mathrm{rAvg}_\lambda} (\mathrm{im} (i)) \subset \mathrm{im}(i)$. This implies that the coboundary map $\delta_{\mathrm{rAvg}_\lambda} : C^n_{\mathrm{rAvg}_\lambda} (A,A,P) \rightarrow C^{n+1}_{\mathrm{rAvg}_\lambda} (A,A,P)$ induces a coboundary map (which we denote by $\delta_{\mathrm{Avg}_\lambda}$)
\begin{align*}
 \delta_{\mathrm{Avg}_\lambda} : C^n_{\mathrm{Avg}_\lambda} (A,P) \rightarrow C^{n+1}_{\mathrm{Avg}_\lambda} (A,P) \text{ for } n \geq 0.
\end{align*}
The map $\delta_{\mathrm{Avg}_\lambda}$ is explicitly given by
\begin{align*}
\delta_{\mathrm{Avg}_\lambda} (f, \gamma) = (\delta_\mathrm{Hoch}^A (f), (-1)^{n-1} d_P (\gamma) + \theta_P (f) ) , \text{ for } (f, \gamma) \in C^n_{\mathrm{Avg}_\lambda} (A,P).
\end{align*}
The cohomology of the cochain complex $\{ C^\bullet_{\mathrm{Avg}_\lambda} (A,P), \delta_{\mathrm{Avg}_\lambda} \}$ is called the cohomology of the relative averaging algebra $(A, P)$ of weight $\lambda$. We denote the corresponding $n$-th cohomology group by $H^n_{\mathrm{Avg}_\lambda}(A, P)$.

\begin{thm}
Let $(A, P)$ be a relative averaging algebra of weight $\lambda$. Then there is a long exact sequence
\begin{align*}
\cdots \longrightarrow H^{n-1}_P (A,A) \longrightarrow H^n_{\mathrm{Avg}_\lambda} (A, P) \longrightarrow H^n_\mathrm{Hoch}(A) \rightarrow H^n_P (A,A) \longrightarrow \cdots .
\end{align*}
Here ${H^n_P (A,A)}$ is the $n$-th cohomology group of the averaging operator $P$ of weight $\lambda$ and $H^n_\mathrm{Hoch}(A)$ is the $n$-th Hochschild cohomology group of the associative algebra $A$.
\end{thm}

\section{Infinitesimal deformations of averaging algebras}\label{sec7}
In this section, we study infinitesimal deformations of a relative averaging algebra of weight $\lambda$. In particular, we show that the set of all equivalence classes of infinitesimal deformations of a relative averaging algebra $(A,B,P)$ of weight $\lambda$ has a bijection with the second cohomology group $H^2_\mathrm{rAvg_\lambda} (A,B, P)$.

Let $(A_\mu, B_\nu^{l,r}, P)$ be a relative averaging algebra of weight $\lambda$. Consider the space $A[[t]]/(t^2)$ (resp. $B[[t]]/(t^2)$) whose elements are of the form $a + t b$, for $a, b \in A$ (resp. $x+ty$, for $x,y \in B$). Then $A[[t]]/(t^2)$ and $B[[t]]/(t^2)$ are both $\mathbb{K}[[t]]/(t^2)$-modules.

\begin{defn}
An {\bf infinitesimal deformation} of a relative averaging algebra $(A_\mu, B_\nu^{l,r}, P)$ of weight $\lambda$ is given by a quintuple $(\mu_t, \nu_t, l_t, r_t, P_t)$ in which
\begin{align*}
    \mu_t = \mu + t \mu_1, \quad \nu_t = \nu + t \nu_1, \quad l_t = l + t l_1, \quad r_t = r+ t r_1  \text{ and } P_t = P+ tP_1
\end{align*}
such that $\big( A[[t]]/(t^2) \big)_{\mu_t}$ is an associative algebra, $\big( B[[t]]/(t^2) \big)_{\nu_t}^{l_t, r_t}$ is an associative  $\big( A[[t]]/(t^2) \big)_{\mu_t}$-bimodule over the ring $\mathbb{K}[[t]]/(t^2)$ and the $\mathbb{K}[[t]]/(t^2)$-linear map $P_t : B[[t]]/(t^2) \rightarrow A[[t]]/(t^2)$ is a relative averaging operator of weight $\lambda$. In other words, $ \big( \big( A[[t]]/(t^2) \big)_{\mu_t},   \big( B[[t]]/(t^2) \big)_{\nu_t}^{l_t, r_t}, P_t     \big) $ is a relative averaging algebra of weight $\lambda$ over the ring $\mathbb{K}[[t]]/(t^2)$.
\end{defn}

Let $(\mu_t, \nu_t, l_t, r_t, P_t)$ be an infinitesimal deformation of a relative averaging algebra $(A,B,P)$ of weight $\lambda$. Then it follows from Theorem \ref{sum-l-inf} that $(\mu_1, \nu_1, l_1 + r_1, P_1) \in Z^2_{\mathrm{rAvg}_\lambda} (A,B,P)$ is a $2$-cocycle.

\begin{defn}
    Let $(\mu_t, \nu_t, l_t, r_t, P_t)$ and $(\mu'_t, \nu'_t, l'_t, r'_t, P'_t)$ be two infinitesimal deformations of a relative averaging algebra $(A, B,P)$ of weight $\lambda$. These two infinitesimal deformations are said to be {\bf equivalent} if there exist linear maps $\varphi_1: A \rightarrow A$ and $\psi_1: B \rightarrow B$ such that the pair $ (\varphi_t = \mathrm{id}_A + t \varphi_1, \psi_t = \mathrm{id}_B + t \psi_1)$  defines a morphism of relative averaging algebras of weight $\lambda$ from
    \begin{align*}
         \big(  \big( A[[t]]/(t^2) \big)_{\mu_t},   \big( B[[t]]/(t^2) \big)_{\nu_t}^{l_t, r_t}, P_t     \big) ~ \text{ to } ~ \big(  \big( A[[t]]/(t^2) \big)_{\mu'_t},   \big( B[[t]]/(t^2) \big)_{\nu'_t}^{l'_t, r'_t}, P'_t \big).
    \end{align*}  .
\end{defn}

It follows from the above definition that $(\mu_t, \nu_t, l_t, r_t, P_t)$ and  $(\mu_t', \nu_t', l_t', r_t', P_t')$ are equivalent if the following conditions are hold modulo $t^2$:
\begin{align*}
    \varphi_t (\mu_t (a, b)) =~& \mu_t' (   \varphi_t (a) , \varphi_t (b)), \\
    \psi_t (\nu_t (x, y) ) =~& \nu_t' ( \psi_t (x), \psi_t (y) ), \\
    \psi_t (l_t (a, x)) =~& l_t' (   \varphi_t (a) , \psi_t (x)),\\
    \psi_t ( r_t (x, a)) =~& r_t' (\psi_t (x), \varphi_t (a)), \\
    \varphi_t \circ P_t =~& P_t' \circ \psi_t,
\end{align*}
for all $a, b \in A$ and $x, y \in B$. Comparing the coefficients of $t$ on both sides of the above identities, we get
\begin{align*}
    \mu_1 - \mu_1' = \delta_\mathrm{Hoch}^A (\varphi_1), \quad & \nu_1 - \nu_1' = \delta_\mathrm{Hoch}^B (\psi_1), \quad (l_1 + r_1) - (l_1' + r_1') = \delta_\mathrm{Hoch}^{\varphi_1, \psi_1} (0),\\
   & P_1 - P_1' =  \theta_P (\varphi_1, \psi_1, 0).
\end{align*}
Thus, we obtain that 
\begin{align*}
     (\mu_1, \nu_1, l_1 + r_1, P_1) - (\mu_1', \nu_1' , l_1' + r_1' , P_1') = \delta_{\mathrm{rAvg}_\lambda} (\varphi_1, \psi_1, 0).
\end{align*}
This shows that the $2$-cocycles  $(\mu_1, \nu_1, l_1 + r_1, P_1)$ and  $(\mu_1', \nu_1' , l_1' + r_1' , P_1')$ are cohomologous. Hence they induce same element in $H^2_{\mathrm{rAvg}_\lambda} ( A, B, P)$. Therefore, we obtain a map
\begin{align*}
    \Lambda : (\text{infinitesimal deformations of }  (A, B, P) )/ \sim ~ \rightarrow H^2_{\mathrm{rAvg}_\lambda} (A,B,P).
\end{align*}
In the following result, we will show that the map $\Lambda$ is a bijection.

\begin{thm}
Let $(A,B,P)$ be a relative averaging algebra of weight $\lambda$. Then there is a bijection between the set of all equivalence classes of infinitesimal deformations of $(A,B,P)$ and the second cohomology groups $H^2_{\mathrm{rAvg}_\lambda} (A,B,P)$.
\end{thm}

\begin{proof}
    Here we will construct the inverse map of $\Lambda$. Let $(\mu_1, \nu_1, \beta_1, \gamma_1) \in Z^2_{\mathrm{aAvg}_\lambda} (A, B, P)$ be a $2$-cocycle. We define maps $l_1 : A \otimes B \rightarrow B$ and $r_1: B \otimes A \rightarrow A$ by
    \begin{align*}
        l_1 (a, x) = \beta_1 (a, x) ~~~~ \text{ and } ~~~~ r_1 (x, a) = \beta_1 (x,a), \text{ for } a \in A, x \in B.
    \end{align*}
    Then it is easy to verify that the quintuple $(\mu_1, \nu_t, l_t, r_t, P_t)$ is an infinitesimal deformation, where
    \begin{align*}
    \mu_t = \mu + t \mu_1, \quad \nu_t = \nu + t \nu_1, \quad l_t = l + t l_1, \quad r_t = r+ t r_1  \text{ and } P_t = P+ t \gamma_1.
\end{align*}
Further, if $(\mu_1, \nu_1, \beta_1, \gamma_1)$ and $(\mu'_1, \nu'_1, \beta'_1, \gamma'_1)$ are two cohomologous $2$-cocycles, say 
\begin{align*}
    (\mu_1, \nu_1, \beta_1, \gamma_1) - (\mu_1, \nu_1, \beta_1, \gamma_1) = \delta_{\mathrm{rAvg}_\lambda} (\varphi_1, \psi_1, 0),
\end{align*}
then the corresponding infinitesimal deformations $(\mu_t, \nu_t, l_t, r_t, P_t)$ and $(\mu'_t, \nu'_t, l'_t, r'_t, P'_t)$ are equivalent by the pair of maps $(\varphi_t = \mathrm{id}_A + t \varphi_1, \psi_t = \mathrm{id}_B + t \psi_1)$. Thus, we obtain a map 
\begin{align*}
\Upsilon : H^2_{\mathrm{rAvg}_\lambda} (A, B, P) \rightarrow (\text{infinitesimal deformations of } (A, B, P) )/ \sim.
\end{align*}
Finally, the maps $\Lambda$ and $\Upsilon$ are inverses to each other. This completes the proof.
\end{proof}

\medskip

\section{Homotopy triassociative algebras and homotopy relative averaging operators of nonzero weight}\label{sec8}

In this section, we first introduce the notion of $Triass_\infty$-algebras (strongly homotopy triassociative algebras) and show that any $Triass_\infty$-algebra naturally gives three $A_\infty$-algebra structures. Next, given an $A_\infty$-algebra and an {$A_\infty$-representation}, we define the notion of a homotopy relative averaging operator of any nonzero weight $\lambda$. Finally, we show that such a homotopy operator naturally induces a $Triass_\infty$-algebra structure (which generalizes Proposition \ref{ravg-ind-triass} in the homotopy context).

The notion of homotopy algebras ($A_\infty$-algebras) first appeared in the work of J. Stasheff in the study of topological loop spaces \cite{stas}. See also \cite{keller,lada-markl} for some details.

\begin{defn}
An {\bf $A_\infty$-algebra} is a pair $(A = \oplus_{i \in \mathbb{Z}} A_i, \{ \mu_k \}_{k=1}^\infty)$ consisting of a graded vector space $A = \oplus_{i \in \mathbb{Z}} A_i$ equipped with a collection $\{ \mu_k : A^{\otimes k} \rightarrow A \}_{k =1}^\infty$ of degree $1$ graded linear maps satisfying the following set of identities:
\begin{align}\label{a-inff}
\sum_{k+l = n+1} \sum_{i=1}^k (-1)^{|a_1| + \cdots + |a_{i-1}|}~ \mu_k \big( a_1, \ldots, a_{i-1} , \mu_l ( a_i, \ldots, a_{i+l-1}), a_{i+l},  \ldots, a_n   \big) = 0,
\end{align}
for all $n \geq 1$ and homogeneous elements $a_1, \ldots, a_n \in A.$
\end{defn}

\begin{defn}
Let $(A, \{ \mu_k \}_{k=1}^\infty)$ be an $A_\infty$-algebra. An {\bf $A_\infty$-representation} of $(A, \{ \mu_k \}_{k=1}^\infty)$ is an $A_\infty$-algebra $(B, \{ \nu_k \}_{k=1}^\infty)$ together with a collection
\begin{align*}
    \{ \eta_k : \mathcal{A}^{k-1,1} \oplus \mathcal{A}^{k-2,2} \oplus \cdots \oplus \mathcal{A}^{1, k-1} \rightarrow B \}_{k=1}^\infty
\end{align*}
of degree $1$ graded linear maps satisfying the identities (\ref{a-inff}) with some (but atmost $n-1$ many) of the elements among $a_1, \ldots, a_n$ is from $B$ and the corresponding operation $\mu_k$ (resp. $\mu_l$) is replaced by $\nu_k$ or $\eta_k$ (resp. $\nu_l$ or $\eta_l$). We often denote an $A_\infty$-representation as above by $(B, \{ \nu_k \}_{k=1}^\infty, \{ \eta_k \}_{k=1}^\infty)$. Like nonhomotopic case, here $\mathcal{A}^{k,l}$ is the direct sum of all possible $(k+l)$ tensor powers of $A$ and $B$ in which $A$ appears $k$ times and $B$ appears $l$ times.
\end{defn}

It follows from the above definition that any $A_\infty$-algebra $(A, \{ \mu_k \}_{k=1}^\infty)$ is obviously an $A_\infty$-representation of itself, where $\eta_k = \mu_k$ for $k \geq 1$. This is called the adjoint $A_\infty$-representation.

\begin{defn}
A {\bf $Triass_\infty$-algebra} is a graded vector space $D= \oplus_{i \in \mathbb{Z}} D_i$ equipped with a collection $\{ \pi_k : \mathbb{K} [T_k] \otimes D^{\otimes k} \rightarrow D \}_{k=1}^\infty$ of degree $1$ graded linear maps satisfying the following set of identities:
\begin{align}\label{tri-inf-iden}
\sum_{k+l = n+1} \sum_{i=1}^k (-1)^{|x_1| + \cdots + |x_{k-1}|} ~ \pi_k \big(  R^{k; i, l}_0 (T) ; x_1, \ldots, x_{i-1}, \pi_l \big( R^{k; i, l}_i (T); x_1, \ldots, x_{i+l-1} ), x_{i+l}, \ldots, x_n \big) = 0,
\end{align}
for all $n \geq 1$, any $T \in T_n$ and homogeneous elements $x_1, \ldots, x_n \in D.$
\end{defn}

\begin{remark}
(i) Let $(D, \{ \pi_k \}_{k=1}^\infty)$ be a $Triass_\infty$-algebra. Since the set $Y_n$ of all planar binary $n$-trees is contained in the set $T_n$ of all $n$-trees, the identities in (\ref{tri-inf-iden}) are naturally held when we consider the trees of $Y_n$. Thus the identities of $Diass_\infty$-algebra (strongly homotopy dissociative algebra) follow. In other words, any $Triass_\infty$-algebra is intrinsically a $Diass_\infty$-algebra. This generalizes the similar result from the nonhomotopic case.
%This generalizes the Remark \textcolor{red}{left} in the homotopy context.

\medskip

(ii) For each $n \geq 1$, the set $T_n$ has three distinguished trees, namely,\\

\begin{center}
\begin{tikzpicture}[scale=0.15]
\draw (0,0) -- (0,-2); \draw (0,0) -- (-3,3); \draw (0,0) -- (3,3); \draw (0.7,0.7) -- (-1.6,3); \draw  (2.3,2.3) -- (1.6,3) ; \draw (1.7,1.7) -- (0.4, 3); \draw (-0.9, 3) -- (-1, 3); \draw (-0.2, 3) -- (-0.3, 3);
\end{tikzpicture} (call it $T^\dashv_n$) \qquad  \quad \begin{tikzpicture}[scale=0.15]
\draw (0,0) -- (0,-2); \draw (0,0) -- (-3,3); \draw (0,0) -- (3,3); \draw (-0.7,0.7) -- (1.6,3); \draw  (-2.3,2.3) -- (-1.6,3) ; \draw (-1.7,1.7) -- (-0.4, 3); \draw (0.9, 3) -- (1, 3); \draw (0.2, 3) -- (0.3, 3);
\end{tikzpicture} (call it $T^\vdash_n$)  \qquad  \quad
 \begin{tikzpicture}[scale=0.15]
\draw (0,0) -- (0,-2); \draw (0,0) -- (-3,3); \draw (0,0) -- (3,3); \draw (0,0) -- (1.5,3); \draw (0,0) -- (-1.5,3); \draw (-0.5,3) -- (-0.4,3); \draw (0.4, 3) -- (0.5,3);
\end{tikzpicture}  (call it $T^\perp_n$).\\
\end{center}
If we apply the maps $d_i : T_n \rightarrow T_{n-1}$ (for any $0 \leq i \leq n$) to the above distinguished trees, one get $d_i (T_n^{\dashv}) = T_{n-1}^{\dashv}$, $d_i (T_n^{\vdash}) = T^\vdash_{n-1}$ and $d_i (T^\perp_n) = T^\perp_{n-1}$. As a result, we have
\begin{align*}
    R_0^{k; i, l} (T_n^\dashv) = T_k^\dashv, ~  R_0^{k; i, l} (T_n^\vdash) = T_k^\vdash, ~ R_0^{k; i, l} (T_n^\perp) = T_k^\perp,\\
    R_i^{k; i, l} (T_n^\dashv) = T_l^\dashv, ~  R_i^{k; i, l} (T_n^\vdash) = T_l^\vdash, ~ R_i^{k; i, l} (T_n^\perp) = T_l^\perp.
\end{align*}
Let $(D, \{ \pi_k \}_{k=1}^\infty)$ be a $Triass_\infty$-algebra. Then it follows that $(D, \{ \pi_k^\dashv \}_{k=1}^\infty)$, $(D, \{ \pi_k^\vdash \}_{k=1}^\infty)$ and $(D, \{ \pi_k^\perp \}_{k=1}^\infty)$ are all $A_\infty$-algebras, where
\begin{align*}
    \pi_k^\dashv = \pi_k|_{T_k^{\dashv} \otimes D^{\otimes k}}, ~~  \pi_k^\vdash = \pi_k|_{T_k^{\vdash} \otimes D^{\otimes k}} ~~ \text{ and } ~~  \pi_k^\perp = \pi_k|_{T_k^{\perp} \otimes D^{\otimes k}}, \text{ for } k \geq 1.
\end{align*}
\end{remark}

\medskip

Let $D = \oplus_{i \in \mathbb{Z}} D_i$ be a graded vector space. For each $n \in \mathbb{Z}$, let $\mathrm{Hom}_n ( \mathbb{K} [T_\bullet] \otimes \overline{T}(D), D)$ be the space of degree $n$ graded linear maps from the graded space
\begin{align*}
     \mathbb{K} [T_\bullet] \otimes \overline{T}(D) = ( \mathbb{K} [T_1] \otimes D) \oplus ( \mathbb{K} [T_2] \otimes D^{\otimes 2}) \oplus \cdots \oplus ( \mathbb{K} [T_n] \otimes D^{\otimes n}) \oplus \cdots
\end{align*}
to the space $D$. Thus, an element $\pi \in \mathrm{Hom}_n  ( \mathbb{K} [T_\bullet] \otimes \overline{T}(D), D)$ is of the form $\pi = \sum_{k=1}^\infty \pi_k$, where $\pi_k : \mathbb{K} [T_k] \otimes D^{\otimes k} \rightarrow D$ (for $k \geq 1$) is a degree $n$ graded linear map. For any 
\begin{align*}
\pi = \sum_{k=1}^\infty \pi_k \in \mathrm{Hom}_m  ( \mathbb{K} [T_\bullet] \otimes \overline{T}(D), D) ~ \text{ and } ~ \varpi = \sum_{l=1}^\infty \varpi_l \in \mathrm{Hom}_n  ( \mathbb{K} [T_\bullet] \otimes \overline{T}(D), D),
\end{align*}
we define an element $\{ \! [ \pi, \varpi ] \! \} \in \mathrm{Hom}_{m+n}  ( \mathbb{K} [T_\bullet] \otimes \overline{T}(D), D)$ by
\begin{align*}
    \{ \! [ \pi, \varpi ] \! \} = \sum_{r=1}^\infty \sum_{k+l=r+1} ( \pi_k \diamond \varpi_l - (-1)^{mn} \varpi_l \diamond \pi_k ),
\end{align*}
where
\begin{align*}
    (\pi_k \diamond \varpi_l) & (T; x_1, \ldots, x_r) \\
    &= \sum_{i=1}^k (-1)^{|x_1| + \cdots + |x_{i-1}|} ~ \pi_k \big( R^{k; i, l}_0 (T); x_1, \ldots, x_{i-1}, \varpi_l \big( R^{k; i, l}_i (T); x_i, \ldots, x_{i+l-1} \big), x_{i+l}, \ldots, x_r   \big),
\end{align*}
for $T \in T_r$ and $x_1, \ldots, x_r \in D$. Then it turns out that $\big( \oplus_{n \in \mathbb{Z}}  \mathrm{Hom}_n  ( \mathbb{K} [T_\bullet] \otimes \overline{T}(D), D) , \{ \! [~,~] \! \}  \big)$ is a graded Lie algebra. Note that an element $\pi = \sum_{k=1}^\infty \pi_k \in \mathrm{Hom}_1  ( \mathbb{K} [T_\bullet] \otimes \overline{T}(D), D)$ is a Maurer-Cartan element of the above graded Lie algebra if and only if the collection $\{ \pi_k: \mathbb{K} [T_k] \otimes D^{\otimes k} \rightarrow D \}_{k=1}^\infty$ defines a $Triass_\infty$-algebra structure on $D$.

Let $(A, \{ \mu_k \}_{k=1}^\infty)$ be an $A_\infty$-algebra and $(B, \{ \nu_k \}_{k=1}^\infty, \{ \eta_k \}_{k=1}^\infty)$ be an $A_\infty$-representation. Then for any nonzero $\lambda$,  the graded vector space $A \oplus B$ can be equipped with a $Triass_\infty$-algebra structure with the operations $\{ \pi_k : \mathbb{K} [T_k] \otimes (A \oplus B)^{\otimes k} \rightarrow A \oplus B \}_{k=1}^\infty$ given by
\begin{align*}
&\pi_k \big(  T; (a_1, x_1), \ldots, (a_k, x_k)  \big) \\
&= \big(  \mu_k (a_1, \ldots, a_k) , \lambda^{l-2} ~(\eta_k + \nu_k) \big(  P(x_1) , \ldots, P(x_{i_1 -1}), x_{i_1}, P(x_{i_1 + 1}), \ldots, x_{i_1 + \cdots + i_{l-1}}, \ldots, P(x_k)   \big) \big),
\end{align*}
when the tree $T \in T_k$ can be uniquely writtens as $T = T^{1} \vee \cdots \vee T^{l }$, for some unique $(i_1 - 1)$-tree $T^{1} \in T_{i_1 - 1}, \ldots, (i_l -1)$-tree $T^{l} \in T_{i_l -1}$.

It follows that the above $Triass_\infty$-algebra $(A \oplus B, \{ \pi_k \}_{k=1}^\infty)$ gives rise to a Maurer-Cartan element $\pi = \sum_{k=1}^\infty \pi_k \in \mathrm{Hom}_1 ( \mathbb{K} [T_\bullet] \otimes \overline{T}(A \oplus B), A \oplus B)$ in the graded Lie algebra 
\begin{align*}
\mathfrak{g} = \big( \oplus_{n \in \mathbb{Z}} \mathrm{Hom}_n ( \mathbb{K} [T_\bullet] \otimes \overline{T}(A \oplus B), A \oplus B)  , \{ \! [ ~,~ ] \! \} \big).
\end{align*}
%The element $\pi$ is explicitly given by $\pi = \sum_{k=1}^\infty (\mu_k + \textcolor{red}{\lambda \nu_k} + \eta_k)$. 
We define a graded subspace $\mathfrak{a} = \oplus_{n \in \mathbb{Z}} \mathrm{Hom}_n ( \mathbb{K} [T_\bullet] \otimes \overline{T}( B), A)$. Then $\mathfrak{a} \subset \mathfrak{g}$ is a graded Lie subalgebra. Let $P: \mathfrak{g} \rightarrow \mathfrak{g}$ be the projection map on the subspace $\mathfrak{a}$. Then $(\mathfrak{g}, \mathfrak{a}, p, \pi)$ is a $V$-data. As a consequence of Theorem \ref{voro-thm}, we get the following result.

\begin{thm}
Let $(A, \{ \mu_k \}_{k=1}^\infty)$ be an $A_\infty$-algebra and $(B, \{ \nu_k \}_{k=1}^\infty, \{ \eta_k \}_{k=1}^\infty)$ be an $A_\infty$-representation. Then the graded vector space $\mathfrak{a} =  \oplus_{n \in \mathbb{Z}} \mathrm{Hom}_n ( \mathbb{K} [T_\bullet] \otimes \overline{T}( B), A)$ inherits a $L_\infty$-algebra structure with the operations $\{ l_k : \mathfrak{a}^{\otimes k} \rightarrow \mathfrak{a} \}_{k=1}^\infty$ are given by
\begin{align*}
    l_k (\gamma_1, \ldots, \gamma_k) = p \{ \! [   \cdots \{ \! [  \{ \! [ \pi, \gamma_1  ] \! \} , \gamma_2 ] \! \}, \ldots, \gamma_k ] \! \},
\end{align*}
for $k \geq 1$ and homogeneous elements $\gamma_1, \ldots, \gamma_k \in \mathfrak{a}$.
\end{thm}

\begin{defn}
Let $(A, \{ \mu_k \}_{k=1}^\infty)$ be an $A_\infty$-algebra and $(B, \{ \nu_k \}_{k=1}^\infty, \{ \eta_k \}_{k=1}^\infty)$ be an $A_\infty$-representation. A homotopy relative averaging operator of weight $\lambda$ is a Maurer-Cartan element of the $L_\infty$-algebra $(\mathfrak{a}, \{ l_k \}_{k=1}^\infty)$.
\end{defn}

It follows that a homotopy relative averaging operator of weight $\lambda$ is an element $P = \sum_{k=1}^\infty P_k \in \mathrm{Hom}_0 ( \mathbb{K}[T_\bullet] \otimes \overline{T}(B), A)$ that satisfies
\begin{align*}
\sum_{k=1}^\infty \frac{1}{k!} l_k (P, \ldots, P) = 0 \quad \text{  equivalently} \quad p \big(  e^{   \{ \! [ -, P] \! \} } \pi  \big) = 0.
\end{align*}

In Proposition \ref{ravg-ind-triass}, we have seen that a relative averaging operator of any nonzero weight $\lambda$ induces a triassociative algebra structure. Here we will generalize this result in the homotopy context.

\begin{prop}
Let $\Pi$ be a homotopy relative averaging operator of weight $\lambda$. Then the graded vector space $B$ inherits a $Triass_\infty$-algebra structure with the operations $\{ \pi_k^P : \mathbb{K} [T_k] \otimes B^{\otimes k} \rightarrow B \}_{k=1}^\infty$ are given by
\begin{align*}
\pi_ k^P (T; x_1, \ldots, x_k ) = \big(   e^{   \{ \! [ -, P ] \! \} } \pi \big) (T; x_1, \ldots, x_k),
\end{align*}
for all $k \geq 1$, $T \in T_k$ and $x_1, \ldots, x_k \in B$.
\end{prop}

\begin{proof}
Since $\{ \! [ \pi, \pi ] \! \} = 0$, it follows that
\begin{align*}
    \{ \! [   e^{   \{ \! [ - , P ] \! \}   }  \pi, e^{ 
   \{ \! [ -, P ] \! \} } \pi  ] \! \} = e^{   \{ \! [ -, P ] \! \} } \{ \! [ \pi, \pi ] \! \} = 0.
\end{align*}
Therefore, $ e^{   \{ \! [ ~, P ] \! \}   }  \pi$ is a Maurer-Cartan element of the graded Lie algebra $\mathfrak{g}$. As a consequence, the collection $\{ \pi_k^P \}_{k=1}^\infty$ defines a $Triass_\infty$-algebra structure on $B$, where $\pi_k^P = (  e^{   \{ \! [ -, P ] \! \}   }  \pi )|_{  \mathbb{K}[T_k] \otimes B^{\otimes k} }$ for $k \geq 1$.
\end{proof}

%***********************************************************************************

%\section{Homotopy triassociative algebras and homotopy averaging algebras of any weight}

%\begin{prop}
%Let $A, B$ be two associative algebras and $B$ be an associative $A$-bimodule. Then the direct sum $A \oplus B$ inherits a triassociative algebra structure with the operations
%\begin{align*}
%(a,x) \dashv (b, y) := (a \cdot_A b,~ x \cdot b), \quad (a,x) \vdash (b, y) := (a \cdot_A b, ~ a \cdot y ) ~~ \text{ and } ~~ (a,x) \perp (b, y) := (a \cdot_A b, \lambda ~ x \cdot_B y),
%\end{align*}
%for $(a,x), (b, y) \in A \oplus B.$ We denote this triassociative algebra by $A \oplus_\lambda B$.
%\end{prop}

%\begin{prop}
%Let $A, B$ be two associative algebras and $A$ acts on the algebra $B$. Then a linear map $P: B \rightarrow A$ is a relative averaging operator of weight $\lambda$ if and only if the graph $Gr (P) = \{ \big( P(x), x \big) | x \in B \}$ is a subalgebra of the triassociative algebra $A \oplus_\lambda B$.
%\end{prop}

\noindent {\bf Acknowledgements.}  Ramkrishna Mandal would like to thank the Government of India for supporting his work through the Prime Minister Research Fellowship. Both authors thank the Department of Mathematics, IIT Kharagpur for providing the beautiful academic atmosphere where the research has been carried out.
%\vspace*{1cm}

\medskip

\noindent {\bf Data Availability Statement.} Data sharing does not apply to this article as no new data were created or analyzed in this study.

\end{document}